\documentclass{amsart}

\usepackage{amsmath}
\usepackage{amssymb}
\numberwithin{equation}{section}

\usepackage{graphicx}

\usepackage{amscd}

\begingroup

\newtheorem{thm}{Theorem}[section]

\newtheorem{cor}{Corollary}[thm]

\newtheorem{defn}[thm]{Definition}

\newtheorem{rmk}[thm]{Remark}

\endgroup

\newcommand{\bpm}{\begin{pmatrix}}
\newcommand{\epm}{\end{pmatrix}}
\newcommand{\beq}{\begin{equation}}
\newcommand{\eeq}{\end{equation}}

\newcommand{\sddots}{\mathinner{\mkern1mu\raise1pt\hbox{.}\mkern2mu
\raise4pt\hbox{.}\mkern2mu\raise7pt\hbox{.}\mkern1mu}}

\newcommand{\psl}{PSL_2(\mathbb{Z})}
\numberwithin{figure}{section}
\newcommand{\calG}{\mathcal G}

\begin{document}

\title[Congruence Subgroups via Graphs on Surfaces]{Congruence and Noncongruence Subgroups of $\Gamma(2)$ via Graphs on Surfaces}

\author{erica j. Whitaker}
\address{Department of Mathematics, Otterbein University, Westerville OH 43081, USA}
\email{ewhitaker@otterbein.edu}

\maketitle
\begin{abstract} 
There is an established bijection between finite-index subgroups $\Gamma$ of $\Gamma(2)$ and bipartite graphs on surfaces, or, equivalently, certain triples of permutations. We utilize this relationship to study both congruence and noncongruence subgroups in terms of the corresponding graphs. We show some elementary criteria which can be used to identify many noncongruence subgroups. Given a graph on a surface, we have a method to produce generators for the corresponding group $\Gamma$ in terms of the generators of $\Gamma(2)$. Given generators for $\Gamma(2n)$, we show how to determine whether or not a graph of level $2n$ corresponds to a congruence subgroup.     
\end{abstract}

\section{Introduction}

Graphs on surfaces have a wide range of applications in mathematics. In particular, the notion of a bipartite graph in which we have a cyclic ordering at each vertex proves surprisingly powerful. There is a well-known correspondence between these graphs and finite-index subgroups $\Gamma$ of $\Gamma(2)$, which can be realized by considering how such groups act on the upper half-plane; this correspondence will be reviewed in Section \ref{background}.
The graphs are easy to describe and work with, even when the properties of the groups are not. One such property is that of congruence. Noncongruence subgroups of the modular group are of interest in number theory through the theory of modular forms and their connections with Galois representations. (See for example the papers of Atkin and Swinnerton-Dyer \cite{atkin}, Li, Long, and Yang \cite{li,li2} and Scholl  \cite{scholl}.) While much is known about congruence subgroups, since one can describe them in terms of congruences on the entries of the matrices, noncongruence subgroups are more mysterious. However, the correspondence between groups and graphs does not discriminate between congruence and noncongruence, so the graphs give a hands-on way to work with both. 

There are several existing tests and criteria for a subgroup of $\Gamma(1)$ to be congruence, such as those due to Hsu \cite{Hsu} and Larcher \cite{larcher}. In Section \ref{gamma1_drawings} we discuss converting these graphs, which are specific to subgroups of $\Gamma(2)$, into the type of graph that can be used to apply Hsu's algorithm. In Section \ref{larchers_results}
 we interpret and generalize some of Larcher's criteria in terms of graphs, and in doing so develop some powerful and almost immediate ways to identify many noncongruence subgroups. 

In Section \ref{relabel_graph} we introduce the notion of a $\Gamma(2)$-tiling. This tool will allow us to produce generators for $\Gamma$ corresponding to a specific graph in terms of the standard generators for $\Gamma(2)$. It will also allow us to give a method by which to determine if one group contains another, given the graph for the large group and generators for the smaller one in terms of the generators for $\Gamma(2)$, by determining if the corresponding graphs cover one another. Thus, given generators for $\Gamma(2n)$, we can determine whether or not a graph of level $2n$ corresponds to a congruence subgroup by determining if the graph for $\Gamma(2n)$ covers the graph for $\Gamma$.   

The tools developed in this paper are especially useful in producing examples of noncongruence subgroups. In a separate paper we will call on these ideas to produce infinite families of noncongruence subgroups of every even level on surfaces of genus 1, and (finite) families of noncongruence subgroups of every allowable even level on surfaces of genus 2.

\section{Background}
\label{background}

\subsection{Subgroups of $PSL_2(\mathbb{Z})$}
\label{subgroups}

We work within the group $PSL_2(\mathbb{Z}) = SL_2(\mathbb{Z})/\pm I$. When we use a matrix, it is always understood as an equivalence class in $PSL_2(\mathbb{Z})$. 

We use the following standard notation and terminology:
\[ \Gamma(n) = \left\{  \gamma \in \psl \mid \gamma \equiv \pm I \pmod{n} \right\}
\]
\[\Gamma_0(n) = \left\{  \gamma = \bpm a & b \\ c & d \epm \in \psl \Bigg| \ c \equiv 0 \pmod{n} \right\}
\]

$\Gamma(n)$ is called the {\em principal congruence subgroup of level n}. A subgroup $\Gamma \subset \psl$ is called {\em congruence} 
\label{level_group} if it contains $\Gamma(n)$ for some $n$. For a congruence subgroup $\Gamma$, we define the {\em level} of $\Gamma$ as the smallest $n$ such that $\Gamma(n) \subset \Gamma$. In particular, we are most concerned with finite-index subgroups of $\Gamma(2)$, which is freely generated by the elements $A = {1 \, 2\choose 0 \, 1}$ and $B ={1 \, 0\choose 2 \, 1}$ (see \cite{birch}). Throughout this paper $A$ and $B$ will always stand for exactly those matrices. 

We are working within $PSL_2(\mathbb{Z})$ instead of $SL_2(\mathbb{Z})$ because we are interested in these groups  acting on the upper half-plane as linear-fractional transformations. Our preferred fundamental domain for $\Gamma(2)$, $\mathcal{D}$, is given in Figure \ref{fig:domain_gamma2}. The dashed lines indicate that the arc from 0 to 1 and the arc from $-1$ to $\infty$ are not included, though in the future this won't be made explicit.
\begin{figure}[htb]
\centering
\includegraphics[width=.45 \textwidth]{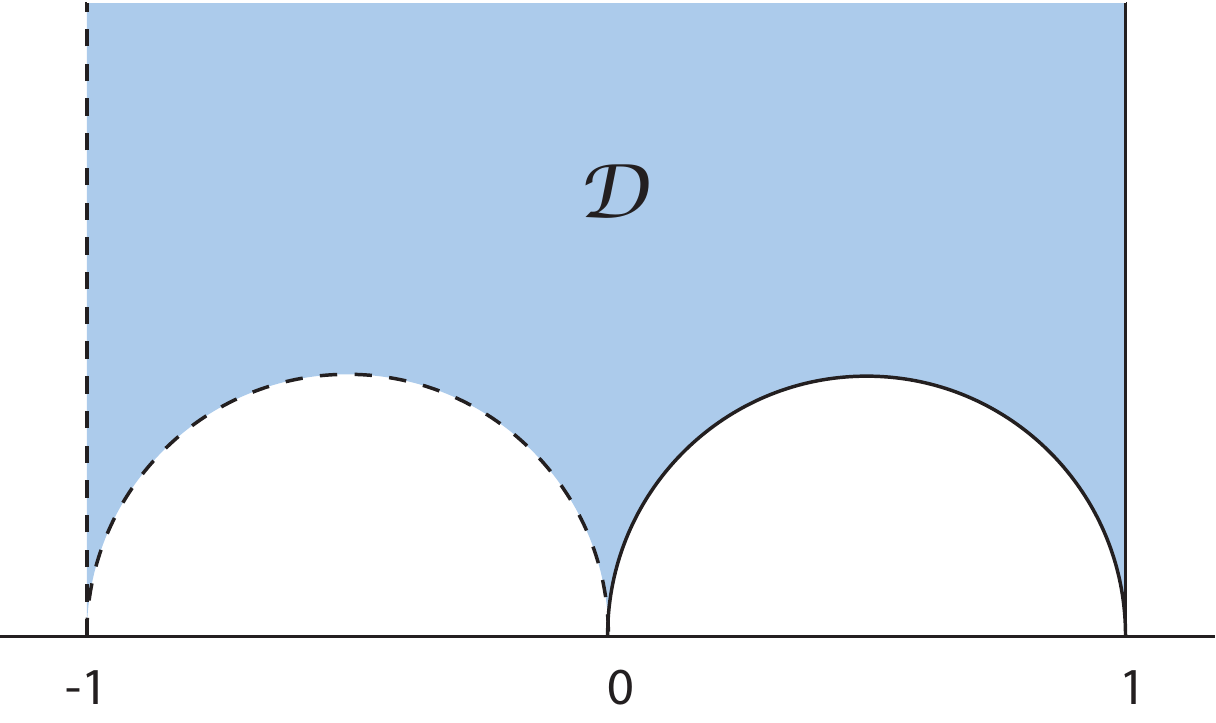}
\caption{Fundamental domain for $\Gamma(2)$}
\label{fig:domain_gamma2}
\end{figure}
 
Given a group $\Gamma \subset \Gamma(2)$ such that $[\Gamma(2):\Gamma]=n<\infty$, we can find a fundamental domain for $\Gamma$ consisting of $n$ copies of the domain for $\Gamma(2)$.  When we give such a domain we will usually show it as tiled by copies of $\mathcal{D}$. We label each tile by an element $\gamma \in \Gamma(2)$ written in terms of the generators $A$ and $B$. The tile labeled as $\gamma$ is the region of the upper half-plane consisting of the image of $\mathcal{D}$ under the matrix $\gamma$. The matrix $A$  acts on $\mathcal{D}$ by translating to the right, which is equivalent to rotating counterclockwise about the cusp $\infty$. If we rotate counterclockwise about an image of $\infty$, we find the adjacent tile by multiplying on the right by $A$. The matrix $B$ acts on $\mathcal{D}$ by rotating clockwise about the cusp 0. If we rotate clockwise about an image of 0, we find the adjacent tile by multiplying on the right by $B$.  Some sample tiles are labeled in Figure \ref{fig:tiles}. Because we are labeling the tiles with coset representatives, we will use $I$ instead of $\mathcal{D}$ for the original domain of $\Gamma(2)$.

\begin{figure}[htb]
\centering
\includegraphics[width=.99 \textwidth]{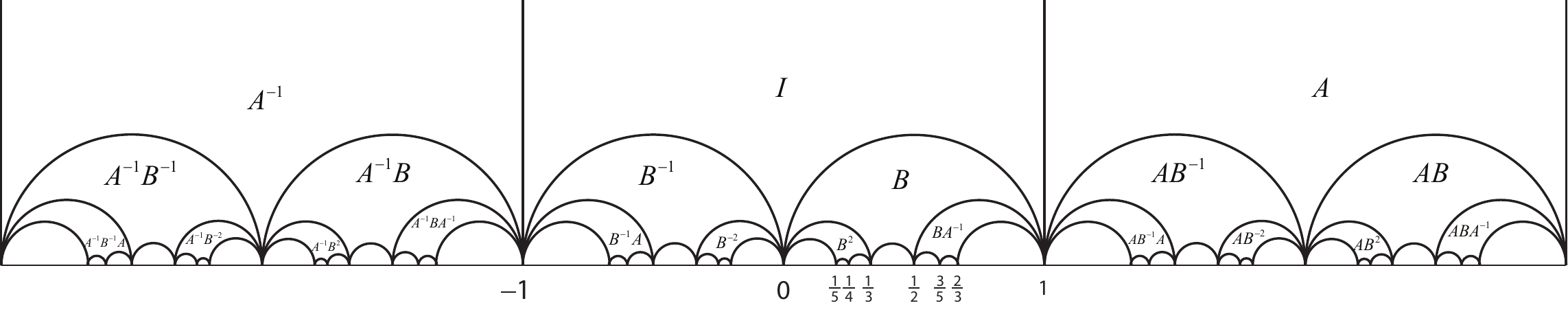}
\caption{Images of $\mathcal{D}$ under some elements of $\Gamma(2)$}
\label{fig:tiles}
\end{figure}

One of the challenges in viewing fundamental domains for groups of higher index is that the cusps become increasingly close together. To overcome this difficulty we will use the following approach: the $x-$values for the cusps will be shown equally spaced on the axis. This results in significant distortion of the regions, but the smallest regions are easier to see. Figure \ref{fig:tiling_rescaled} shows the above tiles displayed in this manner.

\begin{figure}[htb]
\centering
\includegraphics[width=.99 \textwidth]{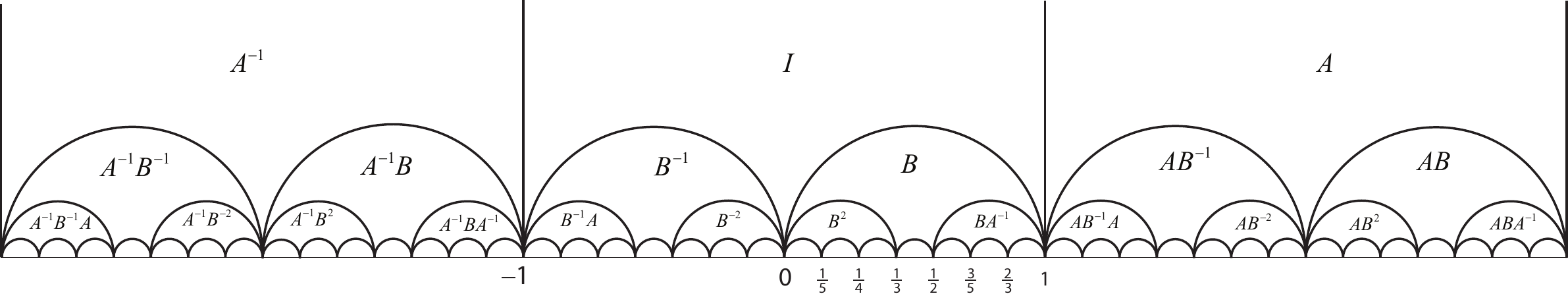}
\caption{Images of $\mathcal{D}$ under some elements of $\Gamma(2)$, rescaled}
\label{fig:tiling_rescaled}
\end{figure}

\subsection{Graphs}
\label{section_graphs}

We will be dealing with a special type of graph  which have been described and named in several equivalent ways. One of the most notable is Grothendieck's term {\em dessins d'enfants}, which he introduced in his {\em Sketch of a Program}, \cite{sketch}. Birch, in \cite{birch}, referred to them as {\em drawings}. In Lando and Zvonkin's text \cite{lando} they are called {\em maps} (or {\em hypermaps} depending on the exact object used). Because these will be the only type of graphs we are interested in, here they will be referred to simply (albeit imprecisely) as {\em graphs}. In this section we will examine these graphs and define other related terms. 

\begin{defn}\rm A {\em graph} $\mathcal G$  will mean a connected bipartite graph $G$ together with a cyclic ordering of the edges at each vertex.
\label{graph_with_ordering_definition}
\end{defn}

Two basic examples are pictured in Figure \ref{fig:first_graph_examples}. We will have a standard convention when labeling the edges: {\em from the viewpoint of a black vertex, the edge label will always lie on the left of the edge.}

 \label{graphs_a_and_b}
\begin{figure}[htb]
\centering
\includegraphics[width=.5 \textwidth]{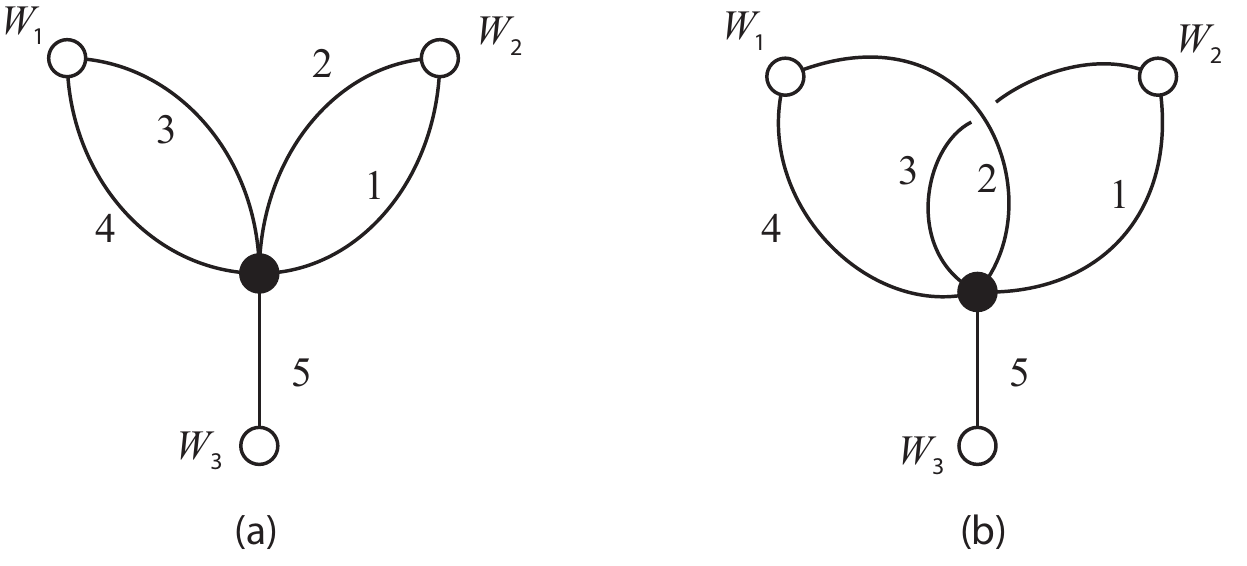}
\caption{Two distinct graphs, each with five edges}
\label{fig:first_graph_examples}
\end{figure}

Here is another way we can view the same objects. That these definitions are equivalent is a consequence of Theorem \ref{correspondence_theorem} below.

\begin{defn}
\label{define_graph_surface}
\rm By the term {\em graph}, we mean a pair $\mathcal G = \{G,\Sigma\}$;  where $G$  is a connected bipartite graph embedded in a compact orientable surface $\Sigma$, and such that the complement of the graph is a disjoint collection of 2-cells, called {\em faces}. We call $\Sigma$ the {\em underlying surface of $\mathcal G$}. If we remove the vertices and one point from the center of each face, we call the resulting $\Sigma '$ the {\em underlying punctured surface of $\mathcal G$}.
\end{defn}

Thinking in these terms, we consider graph (a) to be on a sphere. Graph (b) cannot be placed on a sphere without changing the ordering of the edges. Instead, it can viewed on a torus; see Figure \ref{fig:first_example_torus}. While graph (b) can be embedded on a surface of higher genus, in doing so we would not satisfy the condition that the complement of the graph be a disjoint collection of 2-cells. 
\begin{figure}[htb]
\centering
\includegraphics[width=.7 \textwidth]{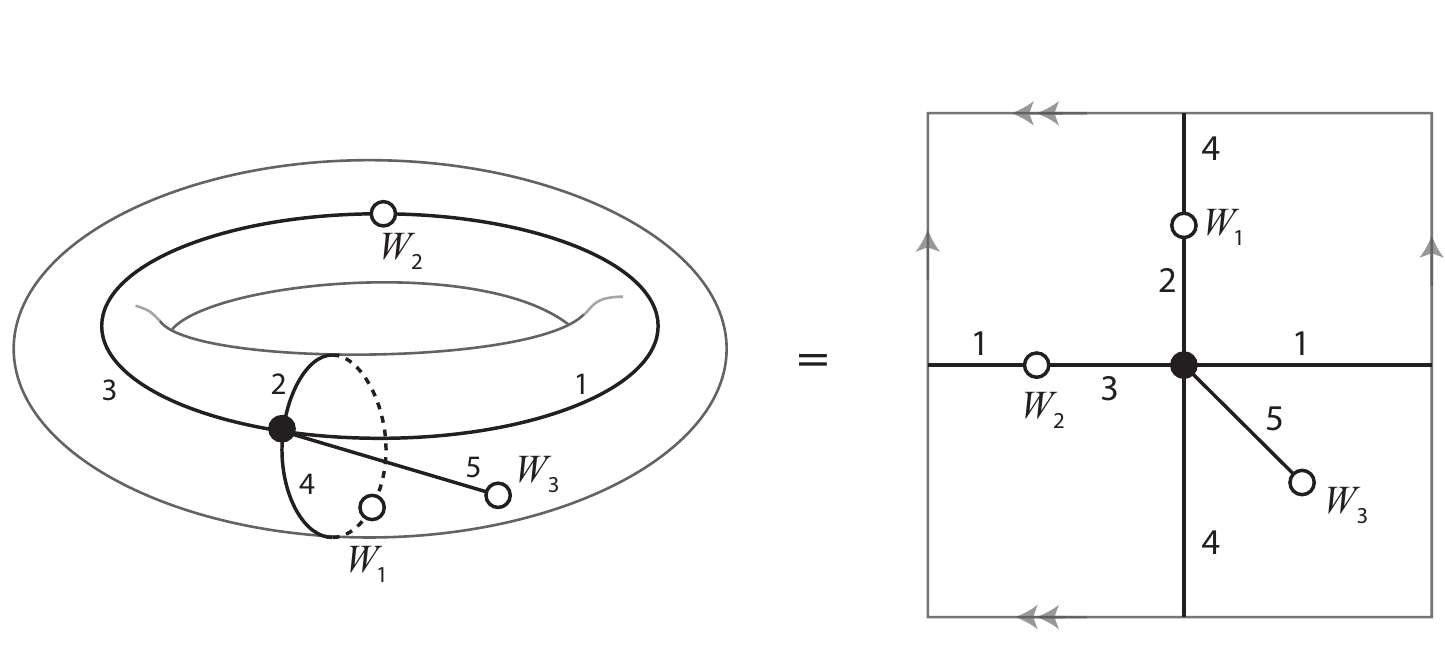}
\caption{Graph (b) viewed on a torus}
\label{fig:first_example_torus}
\end{figure}

We refer to the {\em degree} of a vertex as the number of edges attached to it. For the faces, we say a face {\em contains} an edge if the label for that edge is inside the face, and the {\em degree} of a face is the number of edges it contains. We also have an ordering for each face: when standing at a face center, we rotate counterclockwise and record the edges.

We can use these degrees to define another useful term. 
\label{level_graph}
The {\em level} of a graph is twice the least common multiple of the degrees of all vertices and faces in the graph. Graph (a) has level $2\cdot lcm(1,2,3,5) = 60$, while graph (b) has level $2\cdot lcm(1,2,5) = 20$. (This is related to the term {\em level} we defined in Section \ref{subgroups}, as we will see in Section \ref{cusp_widths}.)

For large graphs it isn't always practical to draw the pictures. Instead, it is enough to keep track of the vertices and the orderings of the edges at each. We can write these orderings as cycles in $S_n$ where $n$ the number of edges in the graph. This leads to another version of our definition, also equivalent because of Theorem \ref{correspondence_theorem}.

\begin{defn}\rm By the term {\em graph}, we mean a pair of permutations in $S_n$: $\sigma$, in which each cycle corresponds to the cyclic ordering of the edges at a black vertex, and $\alpha$, in which each cycle corresponds to the ordering of the edges at a white vertex. (Note that the trivial cycles of length 1 also correspond to vertices.) In order for the graph to be connected, we require that the group generated by $\sigma$ and $\alpha$ be transitive on the $n$ edges. 
\label{permutation_definition}
\end{defn}
In graph (a), at the black vertex we see the permutation $(1,2,3,4,5)$, while the white vertices are represented by $(1,2)(3,4)(5)$. In graph (b) we have the same permutation for the black vertex, but the white vertices yield the permutation $(1,3)(2,4)(5)$. Note that a cycle in the product $\sigma \cdot \alpha$ gives the inverse of the edges we see when rotating counterclockwise within a face. This is shown in Proposition 1.3.16 in \cite{lando}. Thus, we could specify a graph by giving any two of the permutations for the black vertices, white vertices and faces.

\subsection{The Correspondence}
\label{correspondence}

Having introduced both finite-index subgroups of $\Gamma(2)$ and the graphs, we are now ready to understand the correspondence between them. This is found in many places in the literature; here we restate Theorem 1 from Birch \cite{birch}. 

\begin{thm}[(\cite{birch})] For each positive integer $n$, the following families of objects are in $1-1$ correspondence:
\label{correspondence_theorem}
\begin{enumerate}
\item Triples $(\mathcal{R}, \phi, O)$ where $\mathcal{R}$ is an $n-$sheeted Riemann surface, $\phi: \mathcal{R} \rightarrow \overline{\mathbb{C}} = \mathbb{C} \cup \{\infty\}$ is a covering map branched at most above $\{\infty, 0, 1\}$, and $O$ is a point of $\mathcal{R}$ above $\infty$.
\item Quadruples $(\beta, \sigma, \alpha; \star)$ where $\beta$, $\sigma$ and $\alpha$ are permutations of $S_n$ such that $\beta \sigma \alpha = id$ and such that the group generated by $\sigma$, $\alpha$ is transitive on the symbols permuted by $S_n$, and $\star$ is a marked cycle of $\beta$; all modulo equivalence corresponding to simultaneous conjugation by an element of $S_n$.
\item Subgroups $\Gamma \subset \Gamma(2)$ of index $n$, modulo conjugacy by translation.
\item Drawings with $n$ edges.
\end{enumerate}
\end{thm}

Item 2 corresponds with our Definition \ref{permutation_definition}: given $\sigma$ and $\alpha$, we use the relation $\beta \sigma \alpha = id$ to compute the permutation $\beta^{-1}$ of the faces. Then we mark one cycle (i.e., one face); marking a different face amounts to``simultaneous conjugation by an element of $S_n$.'' By ``drawings'' in item 4, he is referring to our Definition \ref{graph_with_ordering_definition}; thus this theorem verifies that these definitions are equivalent.  

\subsection{From a group to a graph}
\label{group_to_graph}

To understand the relationship between the finite-index subgroups of $\Gamma(2)$ and graphs we look first at the domain $\mathcal{D}$ for $\Gamma(2)$ given in Figure \ref{fig:domain_gamma2}. The sides of the domain are identified by the elements $A$ and $B$ of $\Gamma(2)$ to form a sphere with three points removed: the cusps at $1=-1$, 0, and $\infty$. Next we will ``fill the holes'':  at the cusp 0, we add a black vertex; we fill the cusp at 1 with a white vertex, and replace the cusp at infinity with a $*$ to represent a face center, so that we now have a sphere with three marked points. The arc from 0 to 1 will represent an edge; we use dashed lines between white vertices and corresponding face centers. In this way we can identify the group $\Gamma(2)$ with the graph on a sphere consisting of one black vertex, one white vertex, one edge, and one face. See Figure \ref{fig:domain_gamma2_graph}. 
\begin{figure}[htb]
\centering
\includegraphics[width=.7 \textwidth]{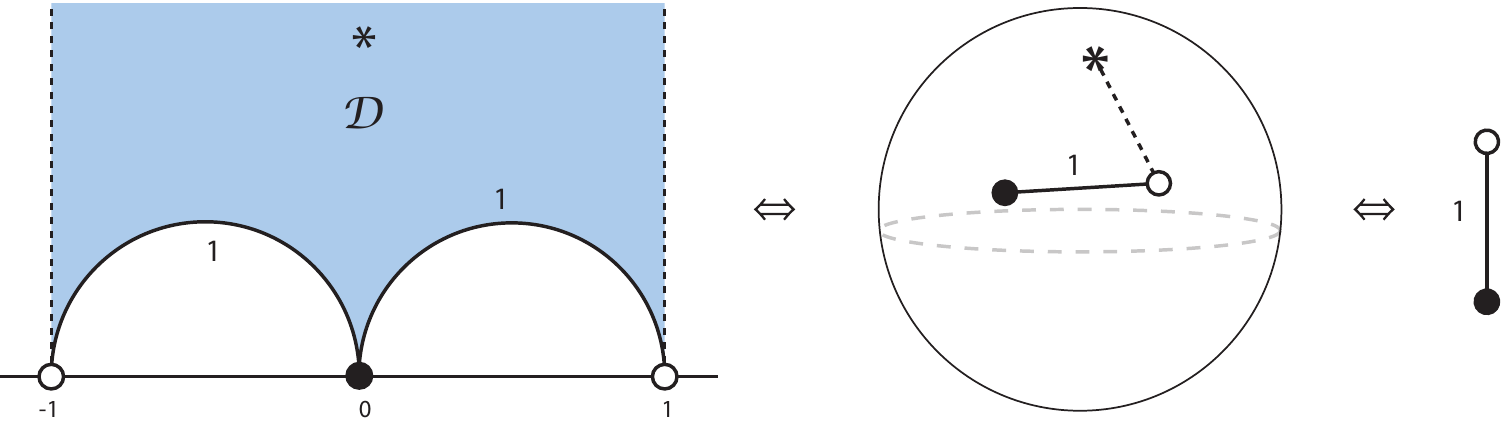}
\caption{$\Gamma(2)$ as a graph}
\label{fig:domain_gamma2_graph}
\end{figure}

Now consider $\Gamma =\Gamma_0(6) \cap \Gamma(2)$. In Figure \ref{fig:domain_gamma06_pairs}, we see a fundamental domain for $\Gamma$. We label 0 and its images under $B$, $B^2$, and $BA^{-1}$ as black vertices; 1 and $-1$ and their images as white vertices, and $\infty$ and its images as face centers with a $*$. The images of the arc from 0 to 1 will be the edges of our graph. We glue the domain to form a surface by finding side-pairing transformations. For example, the sides marked $x_1$ are identified by the element $B^2A^2B^{-1}$ of $\Gamma$.
\begin{figure}[htb]
\centering
\includegraphics[width=.6 \textwidth]{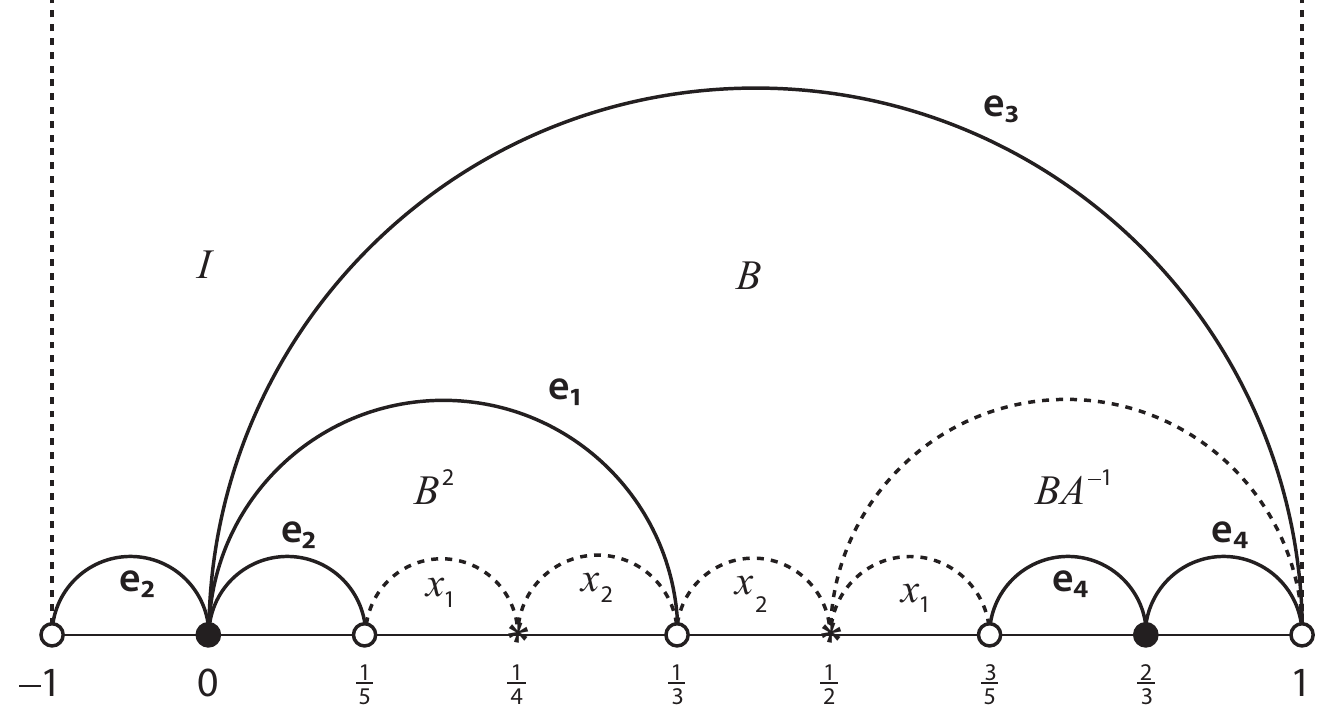}
\caption{Domain for $\Gamma_0(6) \cap \Gamma(2)$ with side pairings}
\label{fig:domain_gamma06_pairs}
\end{figure}

The side-pairing transformations allow us to read off the permutations associated to our graph. We rotate counterclockwise around the black vertices to find $\sigma = (3,2,1)(4)$, the white vertices to find $\alpha = (1)(2, 3, 4)$, and the faces to find $\beta = (1,2,4)(3)$. The Euler characteristic tells us the graph belongs on a sphere; it is pictured in Figure \ref{fig:gamma06_graph}. Recall from Theorem \ref{correspondence_theorem} that we require a marked face $\star$ for our graph; by convention, we always mark the face which has its face center at the cusp $\infty$. 

\begin{figure}[htb]
\centering
\includegraphics[width=.4 \textwidth]{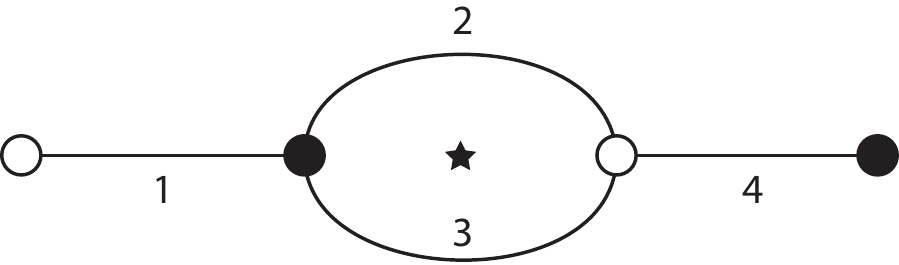}
\caption{Graph associated to $\Gamma_0(6) \cap \Gamma(2)$}
\label{fig:gamma06_graph}
\end{figure}

\subsection{From a graph to a group}
\label{graph_to_group}
Next we will consider the other direction: given a graph, how do we find the corresponding subgroup of $\Gamma(2)$? We can first find a fundamental domain which corresponds to the graph. Then we find the side-pairing transformations, which generate the group. (See Ford \cite{ford}, Theorem 19 in Section 28 or Theorem 10 in Section 32).

For examples we look at the graphs in Figure \ref{fig:first_graph_examples}. First consider graph (b), which has a domain shown in Figure \ref{fig:domain_graph_b}.  The graph has five edges, so a domain uses five copies of $\mathcal{D}$. We can choose $\infty$ to be the center of the only face and the tiles $I$, $A$, $A^2$, $A^3$ and $A^4$ as our domain. We label the edges inside the region according to the cycle for the face, $(1,2,3,5,4)$. The cycle for the black vertex, $(1,2,3,4,5)$ determines our edge pairings. We can verify that with these pairings the cusps at 0, 2, 4 and 8 are all identified to form the black vertex in graph (b), and that rotating around the white vertices we see the cycles $(1,3)(2,4)(5)$.

\begin{figure}[htb]
\centering
\includegraphics[width=.9 \textwidth]{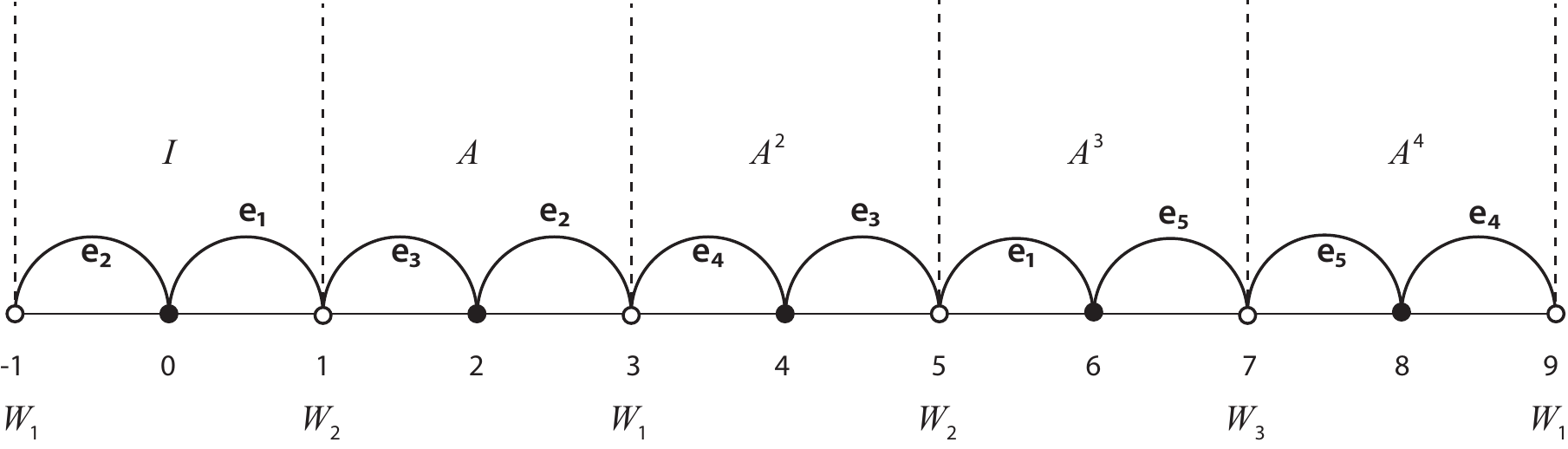}
\caption{Domain for graph (b) from Figure \ref{fig:first_graph_examples}}
\label{fig:domain_graph_b}
\end{figure}  
 
The side-pairing transformations give us generators of the corresponding group $\Gamma_b$. We have six generators: one for each edge, and also $A^5$, which identifies the vertical arcs. For example, for edge 3, we rotate counterclockwise about the black vertex in $A$ to arrive at $A^2$, so the generator $g_3$ satisfies $g_3AB^{-1}=A^2$; thus $g_3 = A^2BA^{-1}$. Similarly we compute the others to get the group 
$$
\Gamma_b = \left<BA^{-3}, \ AB, \ A^2BA^{-1}, \ A^4BA^{-2}, \ A^3BA^{-4}, \ A^5 \right>.
$$

The process for graph (a) is similar. We can choose the degree 3 face $(2,5,4)$ to have its face center at $\infty$, and label the tiles $I$, $A$ and $A^2$ with edges 2, 5, and 4 respectively. We identify the vertical arcs from $-1$ to $\infty$ and 5 to $\infty$, so the element $A^3$ is in the group $\Gamma_a$. 

Now rotate counterclockwise about the black vertex at 0, which has the cycle $(1,2,3,4,5)$. Rotating counterclockwise through edge 2 we expect edge 3, which does not yet appear in our domain, so we add the tile $B^{-1}$. After edge 3 we expect edge 4, so we identify the arc from 0 to $\tfrac{1}{3}$ with the arc from 4 to 5, which adjoins the tile $A^2$. We continue to pair edges and solve for the corresponding generators to get the group
$$
\Gamma_a = \left<A^3, \ B^2A^{-1}, \ A^2B^2, \ ABA^{-2}, \ B^{-1}A^{-1}B, \ AB^{-1}A^{-1}BA^{-1} \right>
$$

\begin{figure}[htb]
\centering
\includegraphics[width=.8 \textwidth]{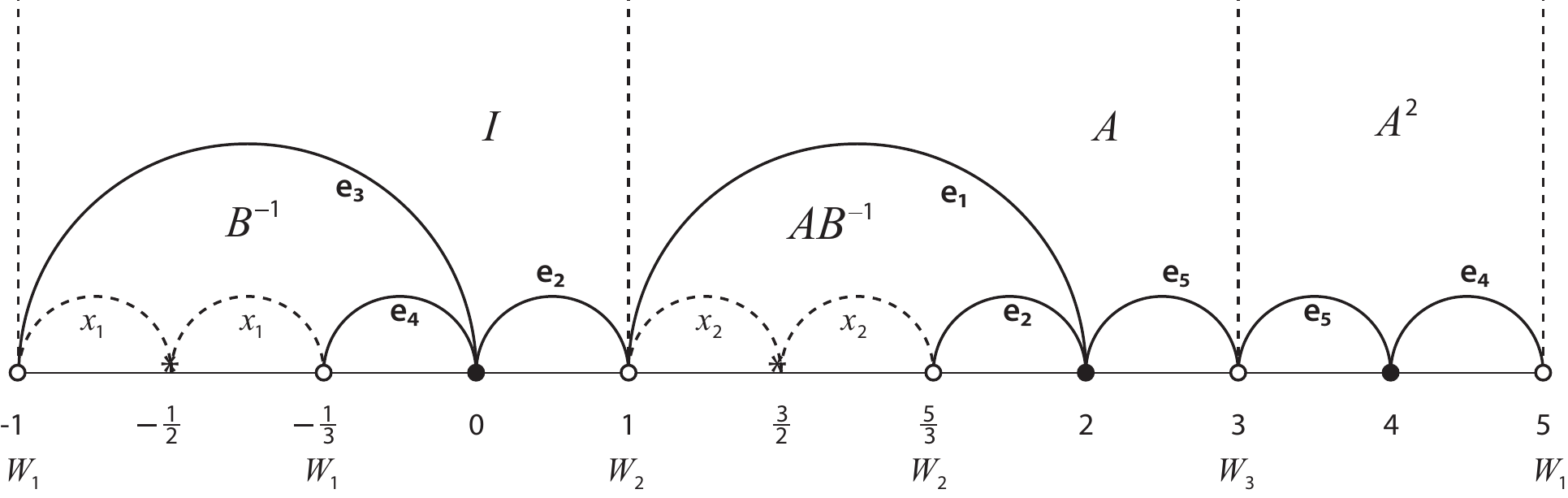}
\caption{Domain for graph (a) from Example \ref{graphs_a_and_b}}
\label{fig:domain_graph_a}
\end{figure}  

The correspondence inspires the following terminology:
\begin{defn}
\label{define_graph_domain}
\rm By a {\em (fundamental) domain of a graph}, we mean a connected fundamental domain of the corresponding group $\Gamma \subset \Gamma(2)$, tiled by copies of $\mathcal{D}$, and with the images of the arc from 0 to 1 labeled according to the edges of the graph. We label the images of 0 as black vertices, the images of 1 and $-1$ as white vertices, and images of $\infty$ as $*$ to represent a face center; when identified these form the black vertices, white vertices and face centers of the given graph.
\end{defn}

\begin{rmk}
\rm Notice that the surface formed by gluing the fundamental domain is the underlying punctured surface $\Sigma '$ of $\mathcal G$, where the punctures are the cusps in the domain. When the cusps are filled with vertices and faces centers as labeled above, we form the underlying compact surface $\Sigma$ of $\mathcal G$.
\end{rmk}

\subsection{Wohlfahrt's Theorem}
\label{cusp_widths}
We have defined the level of a graph as twice the least common multiple of the degrees of all vertices and faces, and we defined the level of a congruence subgroup $\Gamma \subset \Gamma(2)$ as the smallest $n$ such that $\Gamma(n) \subset \Gamma$. In fact, due to a result of Wohlfahrt,  for a congruence subgroup these definitions agree. 

In \cite{W}, Wohlfahrt defines the level of a finite-index subgroup $\Gamma \subset \Gamma(1)$ as the least common multiple of the cusp widths (or amplitudes) for $\Gamma$. For the group $\Gamma(2)$ the cusp width  of $\infty$ is 2, which is the width of its fundamental domain at $\infty$; the degree of the face in the graph for $\Gamma(2)$ is 1.
In the graph for a group $\Gamma \subset \Gamma(2)$, a vertex or face of degree $d$ will touch $d$ copies of the domain for $\Gamma(2)$, and thus have cusp width $2d$. The least common multiple of the cusp widths is the least common multiple of twice each degree, and thus twice the least common multiple of the degrees.

Wohlfahrt then proves that for a group $\Gamma$ of level $n$ in his sense, $\Gamma$ is congruence if and only if $\Gamma$ contains $\Gamma(n)$, and so for congruence subgroups the two definitions agree. Thus, in order to check if a group of level $n$ is congruence, we need check only whether it contains $\Gamma(n)$ for this particular value of $n$.

\section{Drawings for subgroups of $\psl$}
\label{gamma1_drawings}
The idea of looking at subgroups of $\psl$ in terms of drawings is not new. However, for the most part these versions of drawings differ from the graphs we have used. This arises from the fact that others are looking at subgroups of $\Gamma(1) = \psl$ which are not necessarily in $\Gamma(2)$. In this section we explain the relationship between the two types of drawings and the advantages of working with each.

In Figure \ref{fig:domain_gamma2_graph} we showed how $\mathcal{D}$, the fundamental domain of $\Gamma(2)$, can be interpreted as a graph: the arc from 0 to 1 represents an edge, and the cusps at 0, 1 and $\infty$ represent a black vertex, white vertex, and face center, respectively. For subgroups of $\Gamma(2)$, the domains consisting of $n$ copies of $\mathcal{D}$ correspond to graphs with $n$ edges.

In contrast, a standard fundamental domain and corresponding graph for $\Gamma(1)$ is pictured in Figure \ref{fig:gamma1_graph}. We see  only one cusp (at $\infty$). It also has two {\em marked points}, that at $i$, which has an order 2 stabilizer in $\psl$, and $\rho$, whose stabilizer in $\psl$ has order 3. The natural way to associate a graph to this group is to mark the point $\rho$ with a black vertex, $i$ with a white vertex, and continue to mark the cusp $\infty$ with a $*$ to represent a face center. 
\begin{figure}[htb]
\centering
\includegraphics[width=.6 \textwidth]{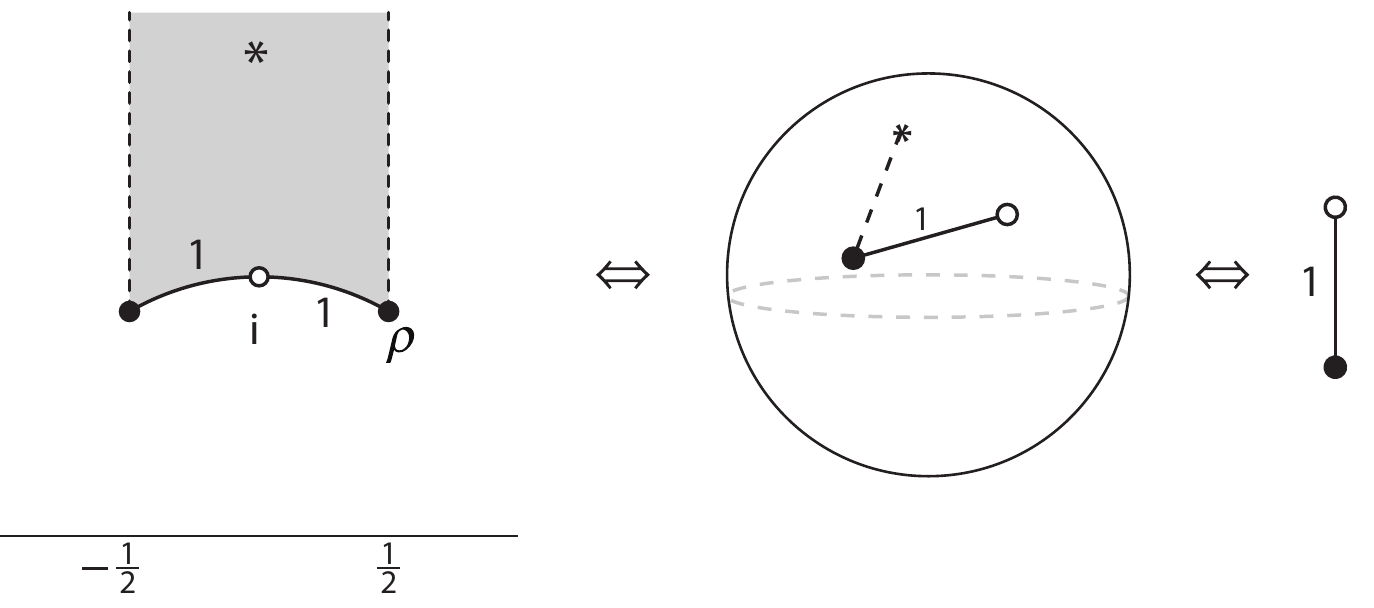}
\caption{Fundamental domain for $\psl$ as a $\Gamma(1)$-graph}
\label{fig:gamma1_graph}
\end{figure}  

Because $[\Gamma(1):\Gamma(2)]=6$, we find a new way to draw a graph for $\Gamma(2)$: we lift the graph for $\Gamma(1)$ to the tiling of a fundamental domain for $\Gamma(2)$ by six copies of our domain for $\Gamma(1)$ as in Figure \ref{fig:gamma2_by_gamma1_graph}. We can extend this for general graphs. For this section we will distinguish these by referring to the {\em $\Gamma(1)$-graph} and the {\em $\Gamma(2)$-graph}, respectively, associated to group $\Gamma$. Notice that in the  $\Gamma(1)$-graph the black vertices all have degree either 1 or 3, and the white vertices have degree 1 or 2. 
\begin{figure}[htb]
\centering
\includegraphics[width=.7 \textwidth]{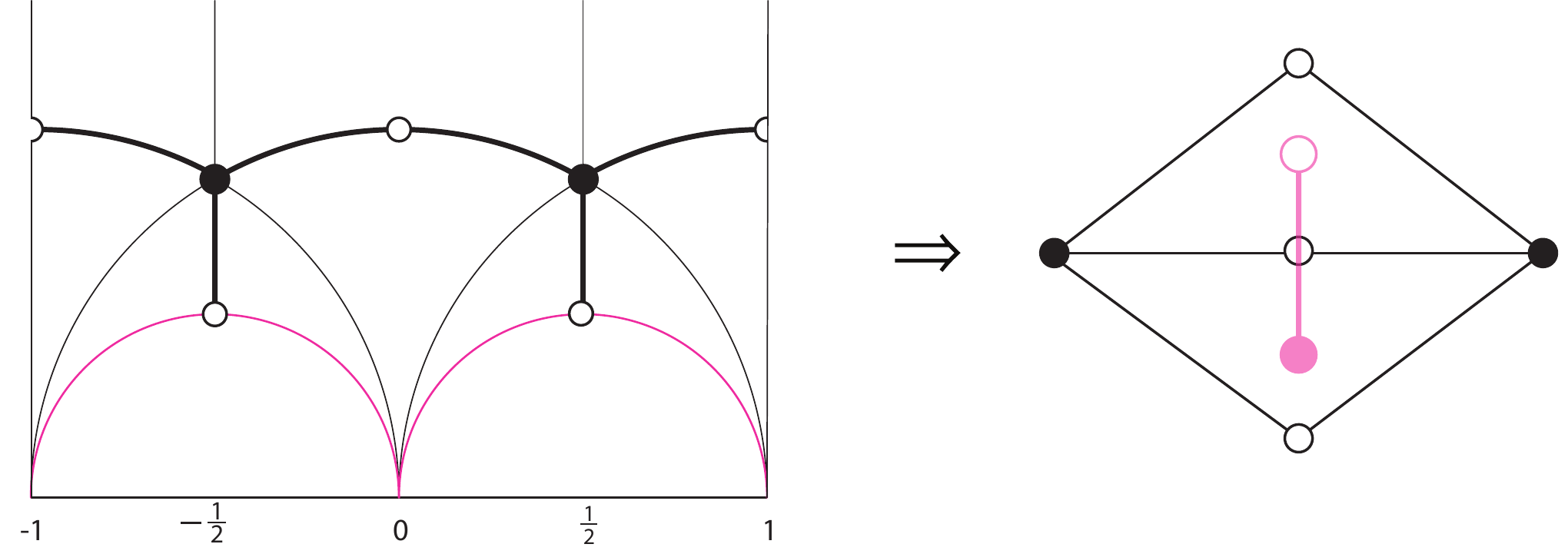}
\caption{Graph for $\Gamma(2)$ as a subgroup of $\psl$}
\label{fig:gamma2_by_gamma1_graph}
\end{figure}  

By computing a few examples one can begin to find the pattern in converting from a $\Gamma(2)$-graph to a $\Gamma(1)$-graph.  In the domains for the graphs for $\Gamma(2)$ the vertices and face centers are all labeled at cusps, so when converting, every black vertex, white vertex and face center will appear as the center of a face on the new graph. As a way to get started, we can notice in Figure \ref{fig:gamma2_by_gamma1_graph} that the original $\Gamma(2)$ edge, i.e., the arc from 0 to 1, now has a white vertex in its center which is connected to two black vertices; so for each edge in our $\Gamma(2)$-graph we can add a white vertex on top of it and two black vertices to the sides. We see some examples of this conversion in Figures \ref{fig:3_star_graphs} and \ref{fig:fish_becomes_spaceship}. 
\begin{figure}[htb]
\centering
\includegraphics[width=.8 \textwidth]{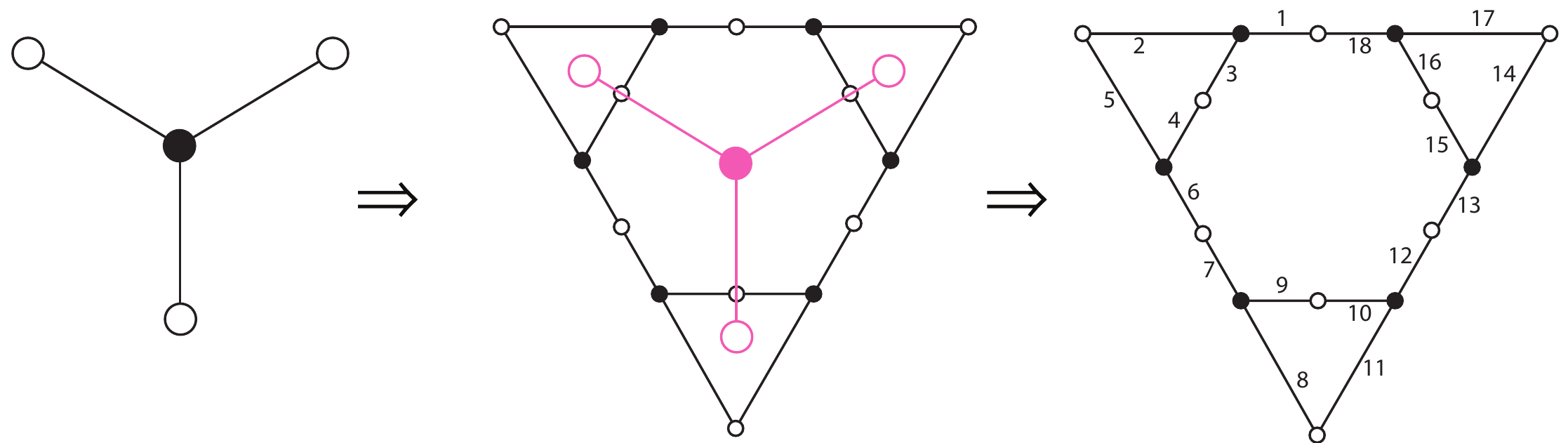}
\caption{Converting the 3-star from a $\Gamma(2)$-graph to a $\Gamma(1)$-graph}
\label{fig:3_star_graphs}
\end{figure}  
\begin{figure}[htb]
\centering
\includegraphics[width=.8 \textwidth]{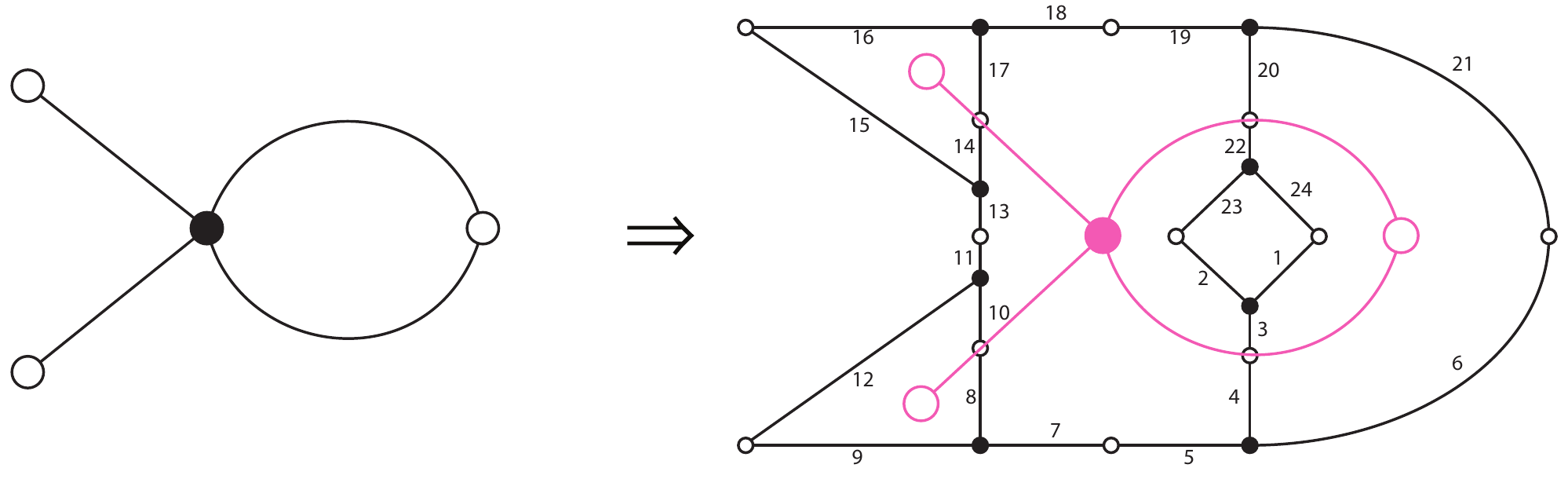}
\caption{A fish in $\Gamma(2)$ becomes a spaceship in $\Gamma(1)$.}
\label{fig:fish_becomes_spaceship}
\end{figure}

The advantage to converting our $\Gamma(2)$-graphs to their $\Gamma(1)$ versions is that there are many results available for subgroups of $\Gamma(1)$. One example of such a result is an algorithm developed by Hsu (\cite{Hsu}) to determine whether a group $\Gamma \subset \psl$ is congruence. He specifies such a $\Gamma$ by giving permutations associated to its $\Gamma(1)$-graph, and his algorithm amounts to checking a list of relations. Thus, if we consider one of our $\Gamma(2)$-graphs, we can determine if the corresponding subgroup of $\Gamma(2)$ is congruence by first converting it to a $\Gamma(1)$-graph and then applying Hsu's algorithm. For example, in applying this algorithm it turns out that the group $\Gamma$ which corresponds to the 3-star is noncongruence. One of the relations that fails amounts to checking that the cycle $\left(\beta_1^4(\alpha_1\sigma_1)^{-4}\beta_1^4\right)^4 \neq 1$, where $\alpha_1$, $\beta_1$ and $\sigma_1$ are the permutations in $S_{18}$ corresponding to the white vertices, black vertices, and faces of the right-most graph in Figure \ref{fig:3_star_graphs}.

In summary, it is true that every finite-index subgroup of $\psl$ can be viewed as a graph. In doing so, we only get restricted types of graphs due to the degree restrictions for the black and white vertices in these $\Gamma(1)$-graphs, but we can then apply known results for subgroups of $\Gamma(1)$ to these graphs. However, there are advantages to working with our versions of the graphs even though they apply only to subgroups of $\Gamma(2)$. Firstly, the $\Gamma(2)$-graphs have six times fewer edges than their $\Gamma(1)$ counterparts, which simplifies computing and allows us to draw graphs for higher-index subgroups. Secondly, we are less restricted on which graphs can occur; in fact by parts 3 and 4 of Birch's Theorem \ref{correspondence_theorem} we can interpret any graph as a $\Gamma(2)$-graph, with no restrictions on the degrees of the black and white vertices. This flexibility can be useful. For example, in a forthcoming paper we will use these graphs to display infinite families of noncongruence subgroups of every possible even level on surfaces of genus 1 and 2.

\section{Criteria for congruence via Larcher}
\label{larchers_results}

In \cite{larcher}, Larcher proves some results about the cusp widths of congruence subgroups $\Gamma \subset \psl$.  In this section we will interpret and generalize two of these results. When interpreted in terms of our graphs, they prove very powerful: for many graphs, they will allow us to determine almost immediately that the corresponding group is not a congruence subgroup.
\begin{thm}[(Larcher \cite{larcher})] 
\label{larcher_1}
If $\Gamma$ is a congruence subgroup of level $m$ and $d$ and $e$ are the respective widths of $\infty$ and $0$ in $
\Gamma$ then $de \equiv 0 \pmod{m}$.
\end{thm}

Recall that in Section \ref{cusp_widths} we saw that a vertex or face of degree $d$ has cusp width $2d$. Thus, we can restate the theorem as follows:

\begin{cor}[(Theorem \ref{larcher_1}, restated)] 
\label{larcher_1_graph}
Let $\Gamma \subset \Gamma(2)$ be a congruence subgroup of level $2n$. In the corresponding graph, let $d$ and $e$ be the respective degrees of the face corresponding to $\infty$ and the black vertex corresponding to $0$. Then $(2d)(2e) \equiv 0 \pmod{2n}$, and thus $2de \equiv 0 \pmod{n}$.
\end{cor}

We can generalize this result as follows:

\begin{thm}
\label{larcher_1_generalized}
Let $\Gamma \subset \Gamma(2)$ be a congruence subgroup of level $2n$. In the corresponding graph, for any face, let $d$ be the degree of that face and $e$ be the degree of any of its vertices. Then $2de \equiv 0 \pmod{n}$.
\end{thm}

\begin{proof}
Because we assumed $\Gamma$ is congruence of level $2n$ we know $\Gamma(2n) \subset \Gamma$, and thus the graph associated to $\Gamma(2n)$ is a cover for the graph associated to $\Gamma$. 

Let $F$ be a face of degree $d$ of the graph associated to $\Gamma$, and $v$ be one of its vertices of degree $e$. Find a fundamental domain for the graph associated to $\Gamma$ as described in Definition \ref{define_graph_domain}. We will see how we can conjugate by an element of $\psl$ to move the cusp associated to $F$ to $\infty$, and then the cusp associated to $v$ to zero. The resulting group, $\Gamma'$, is congruence of level $2n$ if and only if $\Gamma$ is congruence, since $\Gamma(2n) \vartriangleleft \psl$. 

Recall from Theorem \ref{correspondence_theorem} that the correspondence between $\Gamma$ and its graph requires a marked face, which we have chosen to be the face at $\infty$. The cusp that corresponds to the face center of $F$ is the image of $\infty$ under an element $g \in \Gamma(2)$, and so conjugation by $g$ moves the center of $F$ to $\infty$.  This will not change the graph itself, and will not change the degree of the face $F$; the only change is that $F$ is now the marked face. Marking a different face will not change the fact that the graph for $\Gamma(2n)$ is a cover, so this face of degree $d$ and the vertex at 0 must satisfy Larcher's congruence property above.

Now consider the vertex $v$, and let $v'$ be its image under the conjugation in the previous step. Conjugation does not change the degree of a vertex, so the degree of $v'$ is $e$. Notice that because the face $F$ is now centered at $\infty$ it is tiled by translates of $\mathcal{D}$ under $A$, and so the cusps corresponding to black vertices are found at even integers, while the cusps corresponding to the white vertices are found at odd integers. First suppose that the vertex $v$ is a black vertex. Translation by the appropriate power of $A$ will fix $\infty$ and move $v'$ to 0. Recall from part 3 of Theorem \ref{correspondence_theorem} the correspondence between groups and graphs is modulo conjugacy by translation: applying powers of $A$ will not change the graph or the face which is marked, and so the graph is still covered by $\Gamma(2n)$. Hence, Larcher's congruence property is still satisfied.

Suppose instead that $v$ is a white vertex; in this case we can move $v'$ to 0 by conjugating by an odd power of ${1 \, \, \, \, 1\choose 0 \, \,\, \, 1}$. On the level of the graph this has the effect of switching the colors of the black and white vertices. The resulting group $\Gamma'$ is congruence if and only if $\Gamma$ is congruence, because $\Gamma(2n) \vartriangleleft \psl$, and so Larcher's congruence property holds.
\end{proof}
Another result of Larcher can also provide a powerful criteria for congruence:

\begin{thm}[(Larcher \cite{larcher})] 
\label{larcher_2}
If $\Gamma$ is a congruence subgroup of level $m$ then $\Gamma$ contains a cusp of width $m$.
\end{thm}

When interpreted for graphs the statement reads as follows:

\begin{cor}[(Theorem \ref{larcher_2}, restated)] 
\label{larcher_2_graph}
If $\Gamma \subset \Gamma(2)$ is a congruence subgroup of level $2n$, then the graph corresponding to $\Gamma$ has either a vertex or a face of degree $n$.
\end{cor}

These criteria provide necessary conditions for congruence. Consider the example in Figure \ref{fig:fish_becomes_spaceship}: the fish has vertices of degrees 1, 2, and 4 and faces of degree 1 and 3, and thus the level of this graph is $2\cdot lcm(1,2,3,4) = 24$. Because the graph does not contain a vertex or face of degree 12, we immediately conclude that the corresponding group is noncongruence. By contrast, if we were to apply Hsu's algorithm as in the above section, we would need to work with the $\Gamma(1)$ version of the graph (the spaceship). This graph has permutations in $S_{24}$ and the computations are lengthy by comparison.

However, the conditions are not sufficient: notice that while Hsu's algorithm showed us that the 3-star corresponds to a noncongruence subgroup, the graph does not violate Larcher's criteria above. We have an algorithm to determine congruence that depends only on graphs in Section \ref{congruence_from_graphs}.

\section{Tilings and Loops on Graphs}
\label{relabel_graph}

In what follows we describe a way to label a graph, which we will call the {\em $\Gamma(2)$-tiling of $\mathcal{G}$}. In Section \ref{graph_to_group}, when we discussed a way to go from a graph to a group, we saw that the fundamental domain for a graph $\mathcal{G}$ with $n$ edges consists of $n$ copies of $\mathcal{D}$, our fundamental domain for $\Gamma(2)$. We start by reexamining our graph and fundamental domain $\mathcal{D}$ for $\Gamma(2)$ as pictured in Figure \ref{fig:domain_gamma2_graph}. We label the face center at $\infty$ as $*$, and use dotted lines to indicate the arc from 1 to $\infty$ and the arc from $-1$ to $\infty$. Notice that crossing the arc from 1 to $\infty$ corresponds to applying the generator $A$ to $\mathcal{D}$, while crossing the arc from 0 to 1 corresponds to applying $B$ to $\mathcal{D}$. When it will prove helpful, we will indicate these operations in the quadrilateral. See Figure \ref{fig:subdivide_domain}.
\begin{figure}[htb]
\centering
\includegraphics[width=.9 \textwidth]{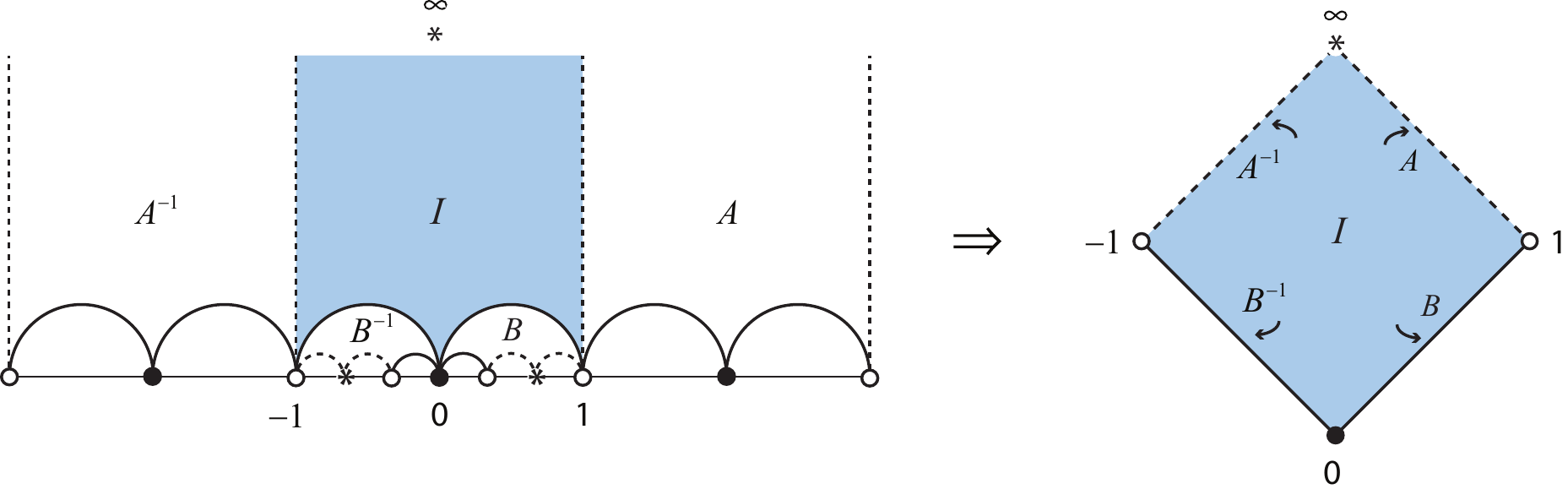}
\caption{Labeling a tile}
\label{fig:subdivide_domain}
\end{figure}

Recall from Section \ref{group_to_graph} that if we are given a group $\Gamma \subset \Gamma(2)$ we can construct the corresponding graph $\mathcal G$ by first forming a fundamental domain for $\mathcal G$. We then identify its sides and fill in its cusps to form a Riemann surface, and draw the graph using $0$ and its images as black vertices, $-1$ and $1$ and their images as white vertices, and images of the arc from $0$ to $1$ as edges. One way we can form a $\Gamma(2)$-tiling for $\mathcal G$ is by taking the above and adding more labels: we explicitly label $\infty$ and its images as $*$, and the images of the lines from $-1$ to $\infty$ and from $1$ to $\infty$ as dotted lines. In doing so, we have tiled the surface $\Sigma'$ with quadrilaterals, each of which corresponds to the tile in Figure \ref{fig:subdivide_domain}. Recall from Part 2 of Theorem \ref{correspondence_theorem} that specifying a graph involves marking one face with a $\star$; this marking will appear in one of the tiles which comprises the corresponding face. (Choosing a different tile within the face corresponds to conjugation by translation, as in Part 3 of the theorem.) Each tile is understood to carry the same labels of $A^{\pm 1}$ and $B^{\pm 1}$ as in Figure \ref{fig:subdivide_domain}, though as the graphs become more complicated these will not always be indicated. When we are specifically referring to the $\Gamma(2)$-tiling which found in this way we will call it the {\em standard $\Gamma(2)$-tiling}.

Suppose instead we are given a graph $\calG$. To form a $\Gamma(2)$-tiling we could first form a fundamental domain as we did in Section \ref{graph_to_group} and then proceed as above. Instead we now describe a procedure to go directly from the graph to a $\Gamma(2)$-tiling without using the fundamental domain as an intermediate step.

First we label each face with a $*$ at its center. Next, we add dotted lines from the white vertices to face centers in such a way that, as we rotate around a white vertex, the solid lines (the graph edges) and the dotted lines (to the face centers) alternate.  This results in a surface tiled with quadrilaterals which is isotopic to the tiling we would have found by first constructing the fundamental domain as above. This isotopy is not unique, but we will fix one for each $\Gamma(2)$-tiling. We will say that two $\Gamma(2)$-tilings are equivalent if they are isotopic.

We have several examples. In Figure \ref{fig:subdivide_graph_gamma2} we see this process applied to the graph corresponding to $\Gamma(2)$ itself. Note that there is only one tile in this situation.
\begin{figure}[htb]
\centering
\includegraphics[width=.4 \textwidth]{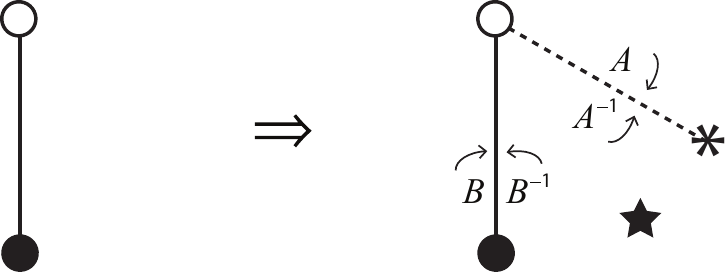}
\caption{$\Gamma(2)$-tiling for the graph associated to $\Gamma(2)$}
\label{fig:subdivide_graph_gamma2}
\end{figure}

Figure \ref{fig:subdivide_3-star} shows the $\Gamma(2)$-tiling of the 3-star, which uses three tiles. 
\begin{figure}[htb]
\centering
\includegraphics[width=.6 \textwidth]{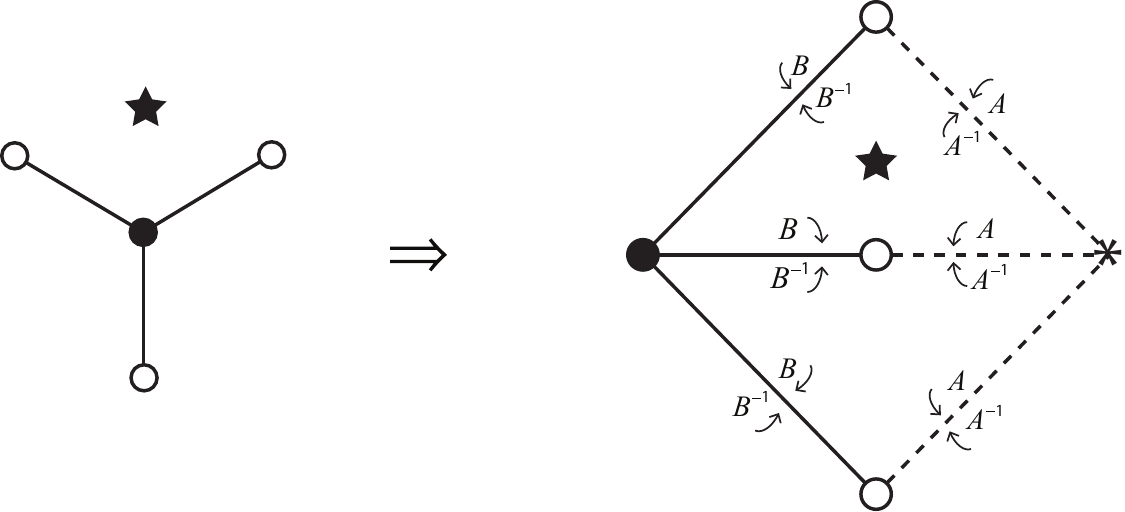}
\caption{$\Gamma(2)$-tiling of 3-star}
\label{fig:subdivide_3-star}
\end{figure}
In Figure \ref{fig:subdivide_graph} we see an example of a tiling for a graph with five edges and two faces.
\begin{figure}[htb]
\centering
\includegraphics[width=.7 \textwidth]{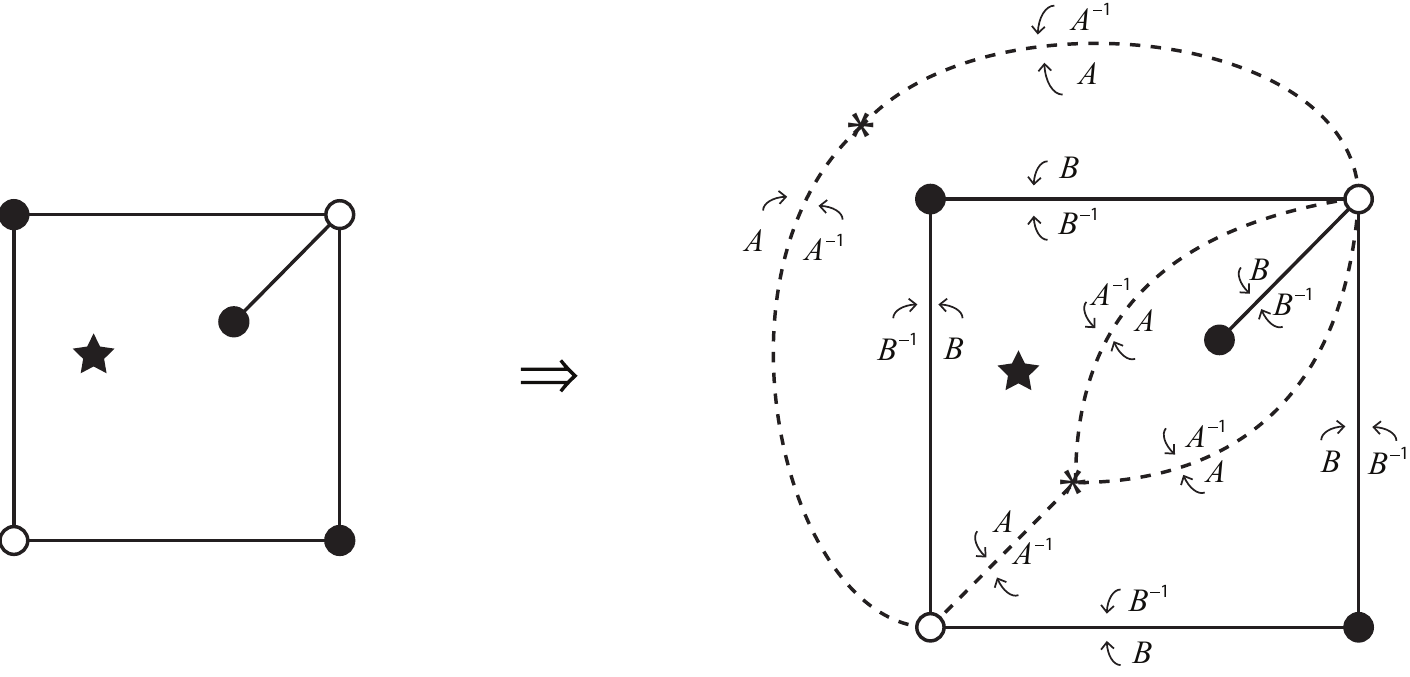}
\caption{$\Gamma(2)$-tiling for a graph  with 5 edges on the sphere.}
\label{fig:subdivide_graph}
\end{figure}
Figure \ref{fig:subdivide_graph_torus} shows an example of a tiling for a graph with four edges which lies on a torus.  
\begin{figure}[htb]
\centering
\includegraphics[width=.8 \textwidth]{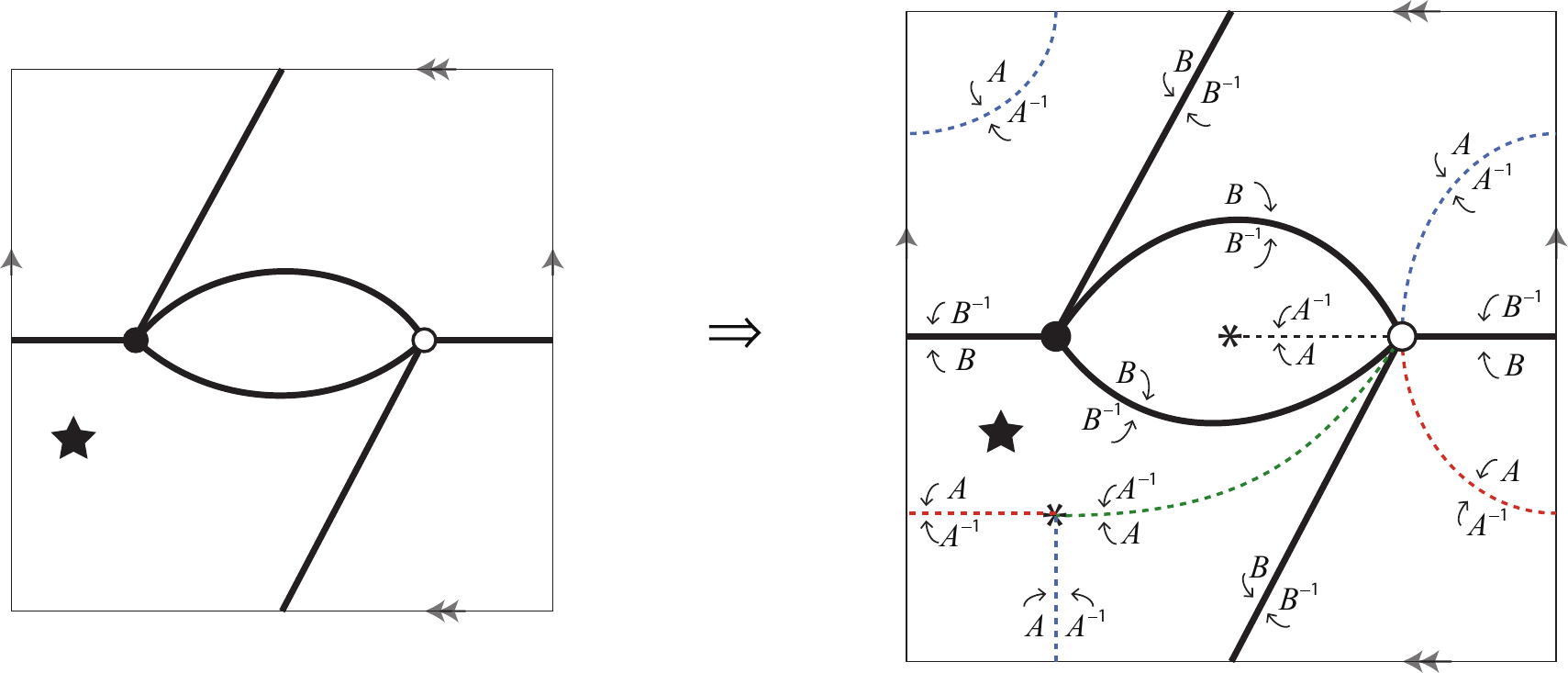}
\caption{$\Gamma(2)$-tiling for a graph with 4 edges on a torus.}
\label{fig:subdivide_graph_torus}
\end{figure}

Now that we have $\Gamma(2)$-tilings, we introduce some useful tools. First, notice the concept of a $\Gamma(2)$-tiling allows us to extend the notion of applying a word $\gamma$ in $A^{\pm 1}$ and $B^{\pm 1}$ to the graph itself in the following way: let $\calG$ be a graph with a $\Gamma(2)$-tiling, let $p$ be a point in the interior of one of the tiles. To apply $\gamma$ to $p$, we isotope the tiling to the standard $\Gamma(2)$-tiling, lift the standard tiling to the upper half-plane, and apply $\gamma$ to the lift of $p$. We then push down to the standard tiling and isotope back to our original $\Gamma(2)$-tiling, labeling the image of this point as $\gamma(p)$. We consider $\gamma(p)$ to be a point on $\calG$ associated to $p$. In fact, $p$ has a unique point associated to it in every tile: two points $p_1$ and $p_2$ are considered {\em associated} if there exists $\gamma \in \Gamma(2)$ so that $\gamma(p_1) = p_2$.

\begin{defn}
\rm
Let $\calG$ be a graph with a $\Gamma(2)$-tiling. A {\em path on a graph $\calG$} is a path on the underlying punctured surface $\Sigma '$ which is transverse to the tiling, begins at the marking $\star$ of the marked face and ends at a point associated to $\star$. 
\end{defn}

Suppose we have a path on $\calG$, a graph with a $\Gamma(2)$-tiling. We can associate this path to a word $\gamma$ in $A^{\pm 1}$ and $B^{\pm 1}$ by following the path starting at the base point $\star$ and recording the labels as we cross each solid or dotted line, always multiplying the labels on the left. (If the path is a loop we must choose a direction. Reversing the direction will form the inverse element $\gamma^{-1}$.) Conversely, suppose we are given a word $\gamma$ in $A^{\pm 1}$ and $B^{\pm 1}$ and want to construct a path. We begin at $\star$ and draw the path by crossing the appropriate solid or dotted line as we read each symbol of $\gamma$ from right to left. The path will end at $\gamma(\star)$.

These processes are essentially inverses: if we begin with a path and find its associated word, and then use this word to draw a path, the resulting path is isotopic to (and thus considered equivalent to) the original path. If we are given a word $\gamma$ and find its associated path, then finding the word associated to this path will recover $\gamma$ (or perhaps $\gamma^{-1}$ if the path were a loop.)

This association between group elements and paths is useful because of the following:

\begin{thm}
\label{loops_are_elements}
Let $\Gamma \subset \Gamma(2)$ be a finite-index subgroup and $\calG$ be its corresponding graph with a $\Gamma(2)$-tiling. Suppose we have $\gamma \in \Gamma(2)$ expressed as a word in $A^{\pm 1}$ and $B^{\pm 1}$ and its associated path on $\calG$. Then $\gamma \in \Gamma$ if and only if its associated path forms a loop.
\end{thm}

\begin{proof}
Let $p$ be a lift of $\star$ in the upper half-plane. By construction, the path associated to $\gamma$ begins at $\star$ and ends at $\gamma(\star)$, where $\gamma(\star)$ is found by applying $\gamma$ to $p$. If $\gamma$ is associated to a loop, then $\gamma(\star) = \star$ on $\calG$. This means that in the upper half-plane, $p$ and $\gamma(p)$ are equivalent under $\Gamma$, so there is a $\gamma' \in \Gamma$ with $\gamma(p) = \gamma'(p)$, and thus $(\gamma')^{-1}\gamma(p)=p$. Because $\Gamma(2)$ acts freely on the upper half-plane, the stabilizer of $p$ is trivial, so $(\gamma')^{-1}\gamma=1$. Thus $\gamma'=\gamma$, so $\gamma \in \Gamma$. Conversely, if the path associated to $\gamma$ is not a loop, then $\gamma(p)$ is not equivalent to $p$ under $\Gamma$, and thus $\gamma \notin \Gamma$.
\end{proof}

\section{Generators from a graph}
\label{generators_from_graph}

Theorem \ref{loops_are_elements} has many useful applications. In this section we will discuss how to use the $\Gamma(2)$-tiling of a graph $\calG$ to read generators for its corresponding group $\Gamma \subset \Gamma(2)$ in terms of $A$ and $B$. 

After forming the $\Gamma(2)$-tiling of $\calG$ we form a collection of loops, each using the $\star$ of the marked face as a base point. The collection will be chosen to generate the fundamental group of the underlying punctured surface $\Sigma'$. For a graph on a sphere, we need $v+f-1$ loops, where $v$ is the total number of vertices and $f$ is the number of faces: we use a loop around all but one of the vertices and face centers. For a graph on a surface of genus $g$ we also add $2g$ loops corresponding to generators of the fundamental group of the underlying compact surface. For each of these loops we then find the associated word $\gamma$ in $A^{\pm1}$ and $B^{\pm1}$; we notate the collection of words as $\mathcal L$. By Theorem \ref{loops_are_elements} these words are all elements of $\Gamma$. The direction we choose to follow each loop is irrelevant, because if $\gamma$ is in a generating set for $\Gamma$, we can replace $\gamma$ with $\gamma^{-1}$ without changing $\Gamma$. In Theorem \ref{loops_are_generators} we will see that $\mathcal L$ is sufficient to generate $\Gamma$. First we will see some examples.

Figure \ref{fig:gamma2_with_loops} shows this process applied to the graph for $\Gamma(2)$ itself. Notice that as we follow the top loop in the counterclockwise direction we find the generator $A$, and following the bottom loop in the clockwise direction produces the generator $B$. This agrees with the fact that $\Gamma(2)$ is generated by the elements $A$ and $B$.
\begin{figure}[htb]
\centering
\includegraphics[width=.17 \textwidth]{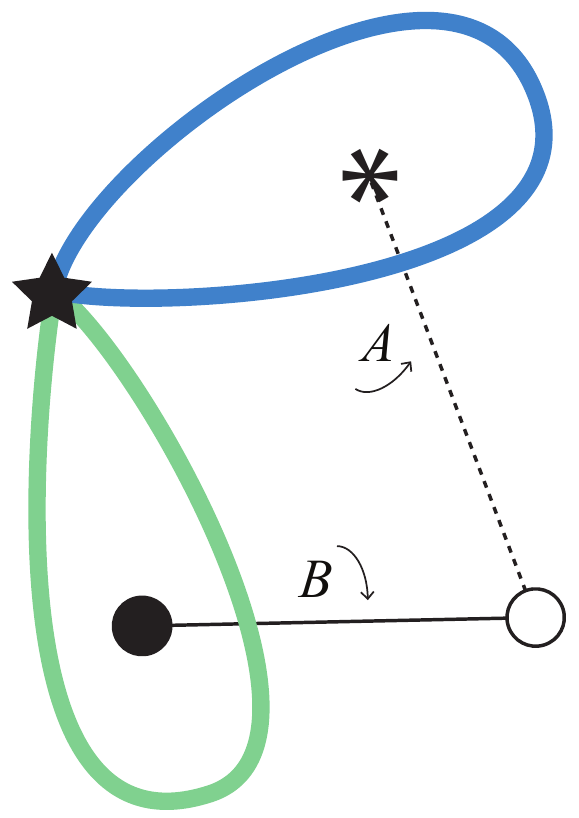}
\caption{Reading generators for $\Gamma(2)$ from its graph}
\label{fig:gamma2_with_loops}
\end{figure}

Now consider the 3-star; we found the $\Gamma(2)$-tiling in Figure \ref{fig:subdivide_3-star}. In Figure \ref{fig:3-star_with_loops} we see this graph with loops added. We can find the generators by following each loop counterclockwise to determine that the group associated to the 3-star is $\Gamma = \left< BA^{-1}, \ A^{-1}B, \ B^{-1}A^{-1}B^2, \ B^{-3} \right>$.

\begin{figure}[htb]
\centering
\includegraphics[width=.5 \textwidth]{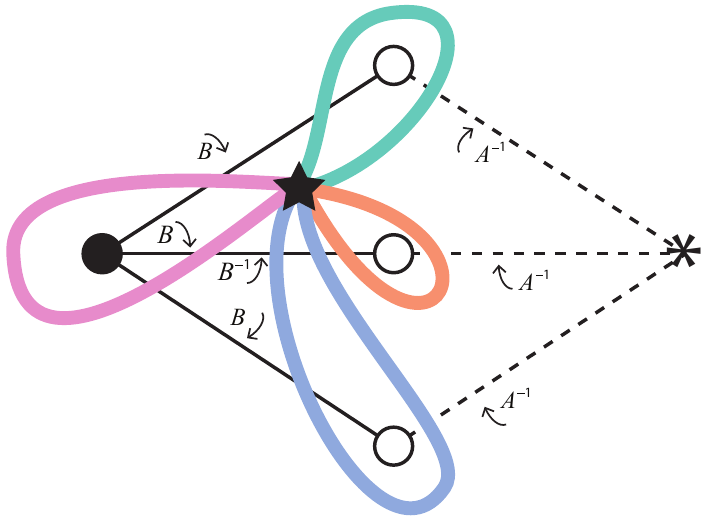}
\caption{Finding generators for the group associated to the 3-star}
\label{fig:3-star_with_loops}
\end{figure}

Using the same process for the graph in Figure \ref{fig:subdivide_graph}, we find the group associated to this graph:
\[
\Gamma = \left< B^{-1}A^2B, \ B^2, \ ABA^{-1}, \ A^3, A^{-1}B^2A, \ B^{-1}AB^{-1}A \right>.
\]
Notice that we have a choice of which vertex or face center to omit when forming loops; in Figure \ref{fig:subdivide_graph} it is natural to avoid the white vertex in the upper-right because it has the highest degree.

\begin{figure}[htb]
\centering
\includegraphics[width=.5 \textwidth]{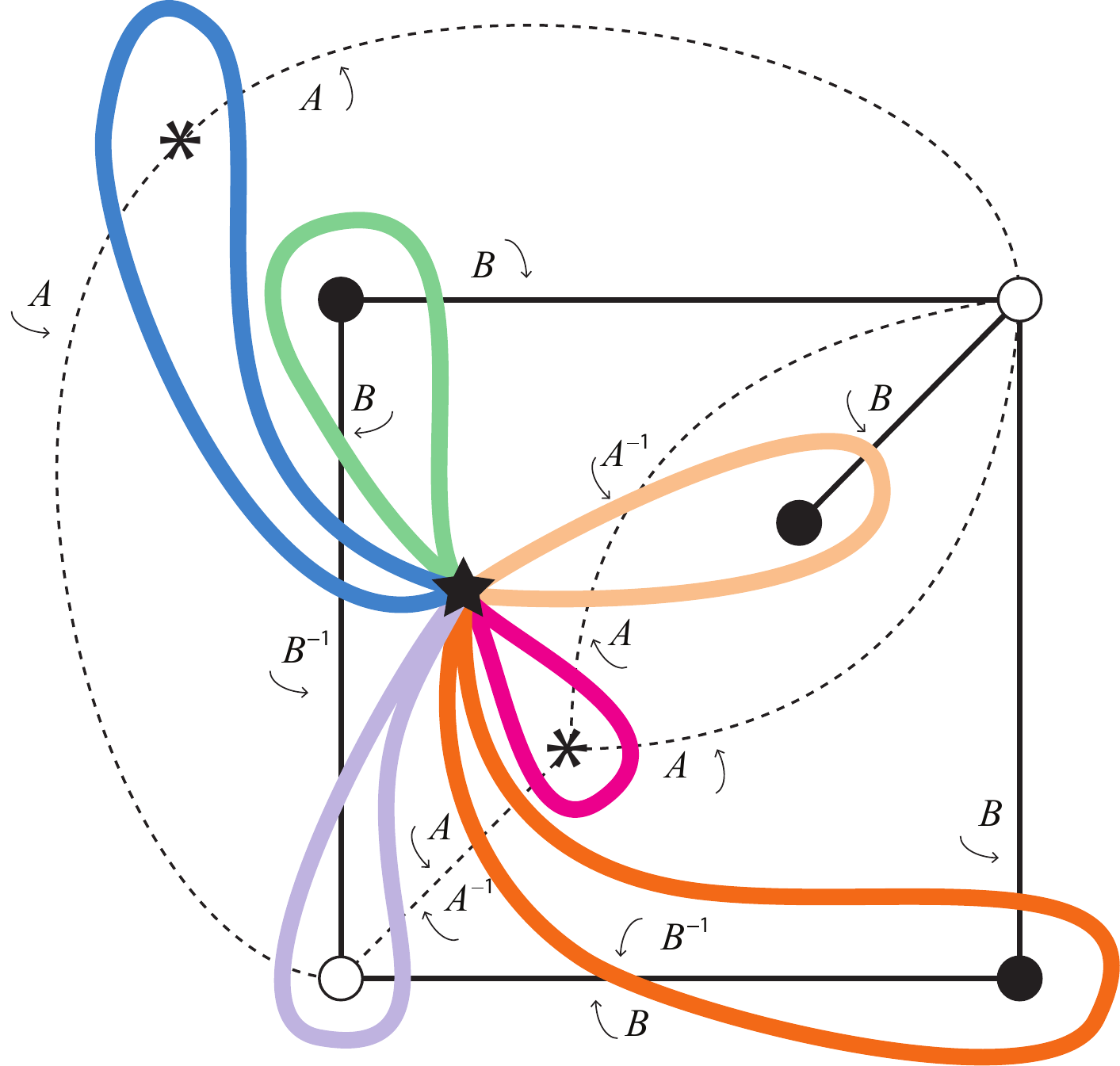}
\caption{Finding generators for the graph in Figure \ref{fig:subdivide_graph}.}
\label{fig:sphere_with_loops}
\end{figure}

We next look at the graph on a torus from Figure \ref{fig:subdivide_graph_torus}. In Figure \ref{fig:torus_with_loops} we see three loops around cusps, and two loops (vertical and horizontal) which generate the fundamental group of a torus. 
\begin{figure}[htb]
\centering
\includegraphics[width=.5 \textwidth]{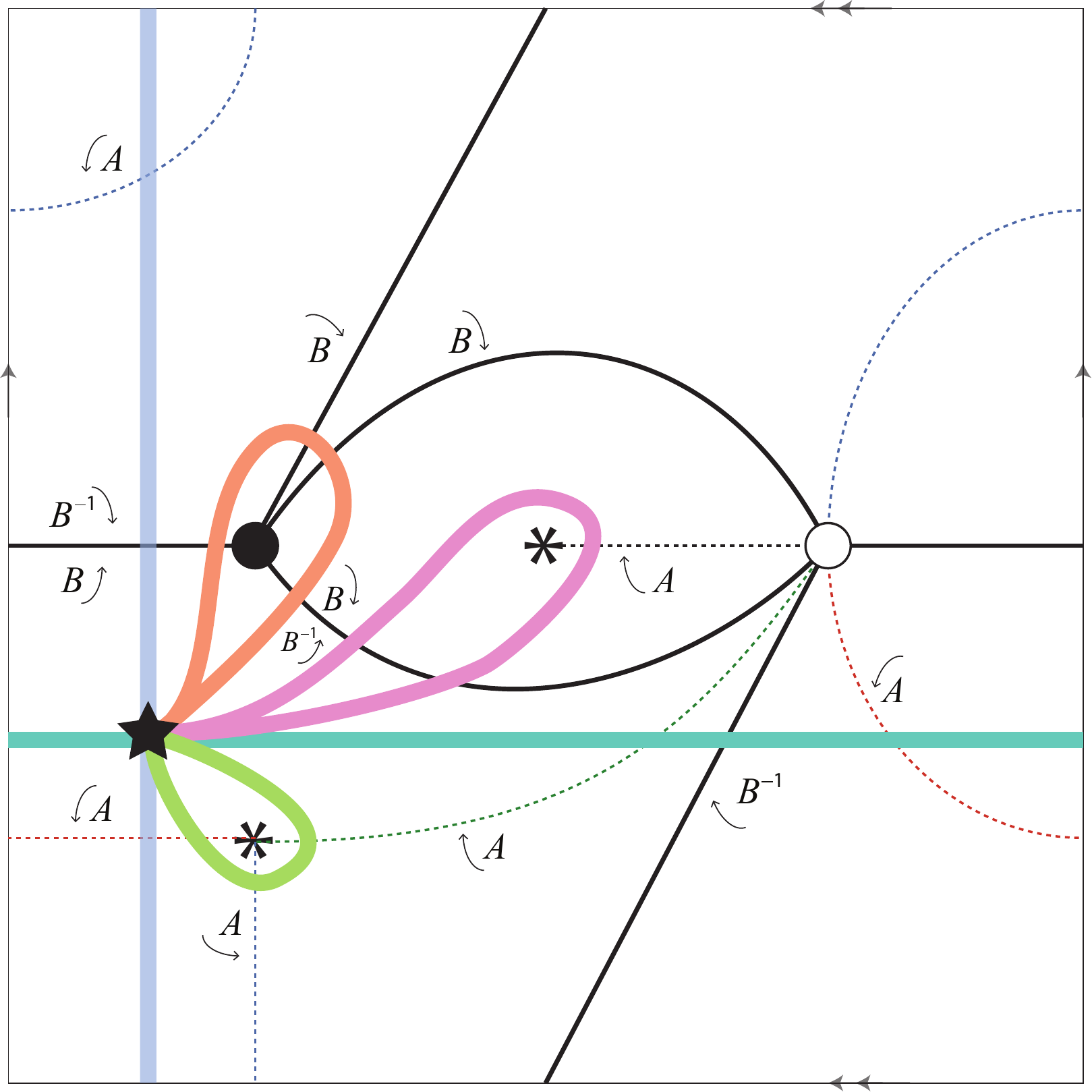}
\caption{Finding generators for the graph in Figure \ref{fig:subdivide_graph_torus}.}
\label{fig:torus_with_loops}
\end{figure}
The group associated to this graph is 
$
\Gamma = \left< B^4, \ BAB^{-1}, \ A^3, \ B^{-1}A^2, \ AB^{-1}A \right>.
$

For another example, we use the graph on a surface of genus 2 pictured in Figure \ref{fig:level12example}.
\begin{figure}[htb]
\centering
\includegraphics[width=.5 \textwidth]{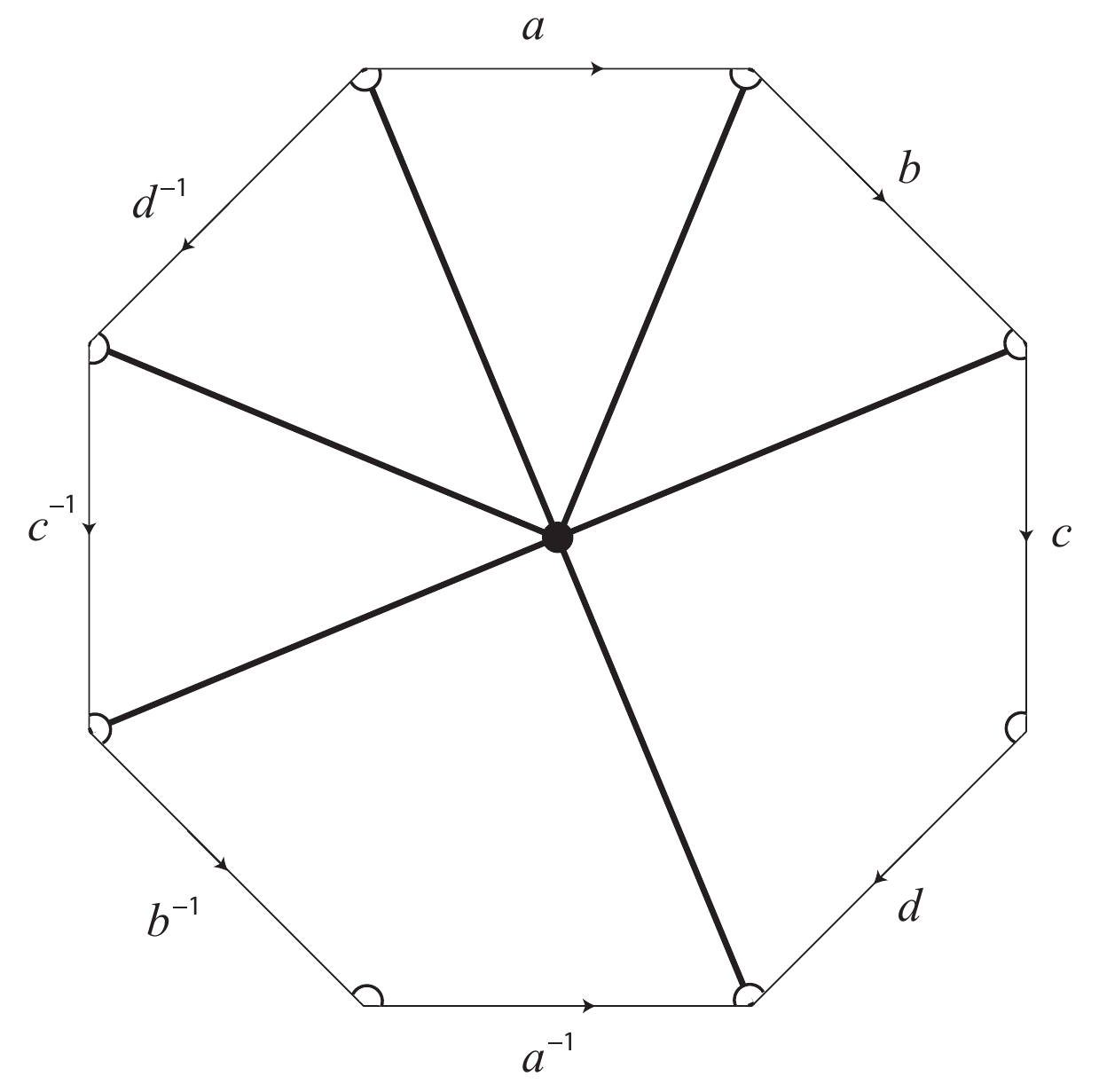}
\caption{A graph of level 12 on a surface of genus 2.}
\label{fig:level12example}
\end{figure}
Its $\Gamma(2)$-tiling is shown in Figure \ref{fig:level12example_with_generators}, along with loops for generators. The labels on the tiles are not shown explicitely, though the conventions are the same. The group associated to this graph is 
$
\Gamma = \left< B^6, \ B^{-1}A^3B, \ A^3, \ B^2A^{-1}B, \ BAB, \ AB^2, \ B^3A \right>.
$
\begin{figure}[htb]
\centering
\includegraphics[width=.5 \textwidth]{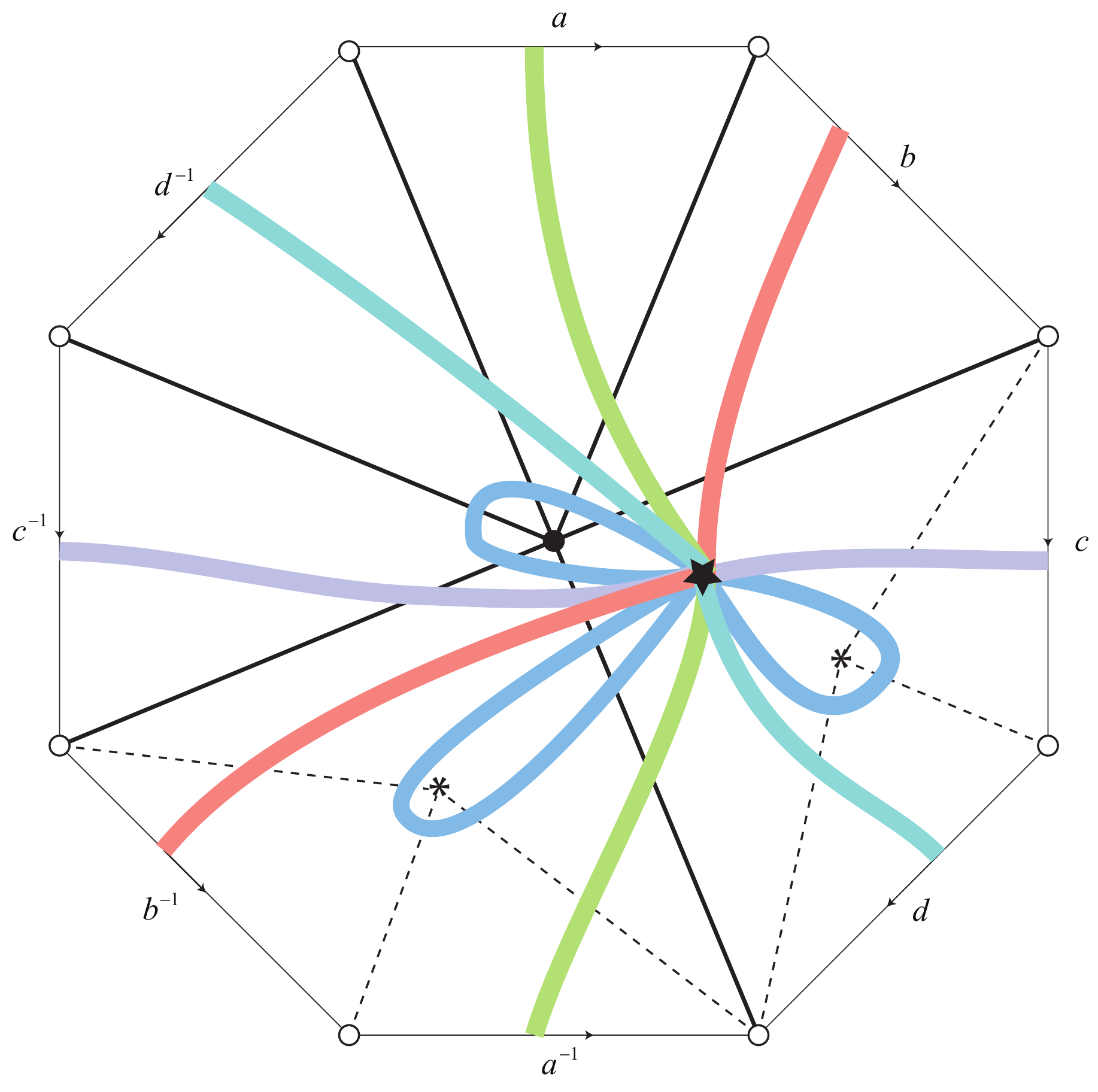}
\caption{Finding generators for the graph in Figure \ref{fig:level12example}}
\label{fig:level12example_with_generators}
\end{figure}

Having seen several examples, we now justify our prodcedure:
\begin{thm} A finite-index subgroup $\Gamma \subset \Gamma(2)$ can be generated by the words in $\mathcal L$ found from the loops drawn as above for its associated graph.
\label{loops_are_generators}
\end{thm}
\begin{proof}
We see this by understanding the how the group $\Gamma$ is associated to a graph $\mathcal{G}$. First, refer to Section \ref{graph_to_group}: given a group $\Gamma$, we find the associated graph by first finding a fundamental domain for $\Gamma$, and then finding side-pairing transformations to form a punctured surface $\Sigma'$. Notice that the punctures are at the vertices and face centers of the graph.  The side-pairing transformations generate $\Gamma$, which is the fundamental group of this punctured surface because $\Sigma' \simeq \Gamma\backslash \mathbb{H}^+$, with $\Gamma$ acting freely on the upper-half plane $\mathbb{H}^+$.

We can also compute the fundamental group of $\Sigma'$ using algebraic topology. Given a surface $\Sigma'$ of genus $g$ with $r$ punctures, $\pi_1(\Sigma')$ is a free group on $2g+r-1$ generators. The generators are homotopy classes of loops; $2g$ for the genus of the surface, and loops around all but one of the punctures. These are the loops from which we formed our collection of words $\mathcal L$, each of which is associated to an element of $\Gamma$ as in Theorem \ref{loops_are_elements}. The operation of concatenation on the loops is compatible with the operation of multiplying the group elements, so the group generated by $\mathcal L$ is isomorphic to the fundamental group of the surface.  

Thus, though the generators we have found in $\mathcal L$ might not directly correspond to side-pairing transformations of a fundamental domain for $\Gamma$, they still generate the fundamental group $\Gamma$ of the punctured surface.
\end{proof}

\section{Using graphs to determine congruence}
\label{congruence_from_graphs}

Suppose we have a graph $\mathcal{G}_1$ corresponding to a finite-index subgroup $\Gamma_1$ of $\Gamma(2)$. Now suppose we have another finite-index subgroup $\Gamma_2 \subset \Gamma(2)$ for which we know generators in terms of $A$ and $B$. In this section we will show how to determine whether $\Gamma_2 \subset \Gamma_1$. One useful application of this result is to determine whether a group is congruence: if $\Gamma_1$ has level $2n$ and we know generators for $\Gamma(2n)$, we can determine whether $\Gamma(2n) \subset \Gamma_1$. 

The procedure is as follows: find a $\Gamma(2)$-tiling for $\mathcal{G}_1$ as we did in Section 
\ref{relabel_graph}. Let $\gamma$ be an element of $\Gamma_2$ expressed as a word in $A^ {\pm 1}$ and $B^{\pm 1}$, and find its associated path on $\calG_1$. By applying Theorem \ref{loops_are_elements} we can determine whether $\gamma$ is in $\Gamma_1$ by seeing if the path forms a loop. By testing the elements of a generating set for $\Gamma_2$, we can determine whether or not $\Gamma_2 \subset \Gamma_1$.

We will look at some examples. Consider the graph in Figure \ref{fig:gamma06_graph}. This graph corresponds to $\Gamma_0(6) \cap \Gamma(2)$, which we know contains $\Gamma(6)$. If we were handed the graph out of context we could first compute that its level is 6 by finding the least common multiple of the degrees of its vertices and faces as we did in Section \ref{section_graphs}. Next we check that the corresponding subgroup is congruence by testing a set of generators for $\Gamma(6)$ to see if the associated paths form loops on our graph. The following generating set for $\Gamma(6)$ was found by constructing a fundamental domain: 
\[
\begin{matrix} A^3, \ B^3, \ ABA^{-2}B^{-2}, \ ABAB^{-2}, \ A^2B^2A^{-1}B^{-1}, \ A^2B^2A^2B^{-1}, \ A^2BA^{-2}B^{-2}
A^{-1}, \ A^2B^3A^{-2}, 
\\ A^2BAB^{-2}A^{-1}, \ AB^{-3}A^{-1}, \ A^2BA^{-1}B^{-1}AB^{-1}, \ ABA^{-1}BAB^{-1}, \ A^2BA^{-1}BAB^{-1}A^{-1}.
\end{matrix}
\]
In Figure \ref{fig:gamma_06_tiled} we see the path corresponding to the generator $ABA^{-2}B^{-2}$, which forms a loop. By checking all thirteen of the above generators, we can confirm that the graph represents a congruence subgroup. 
\begin{figure}[htb]
\centering
\includegraphics[width=.6 \textwidth]{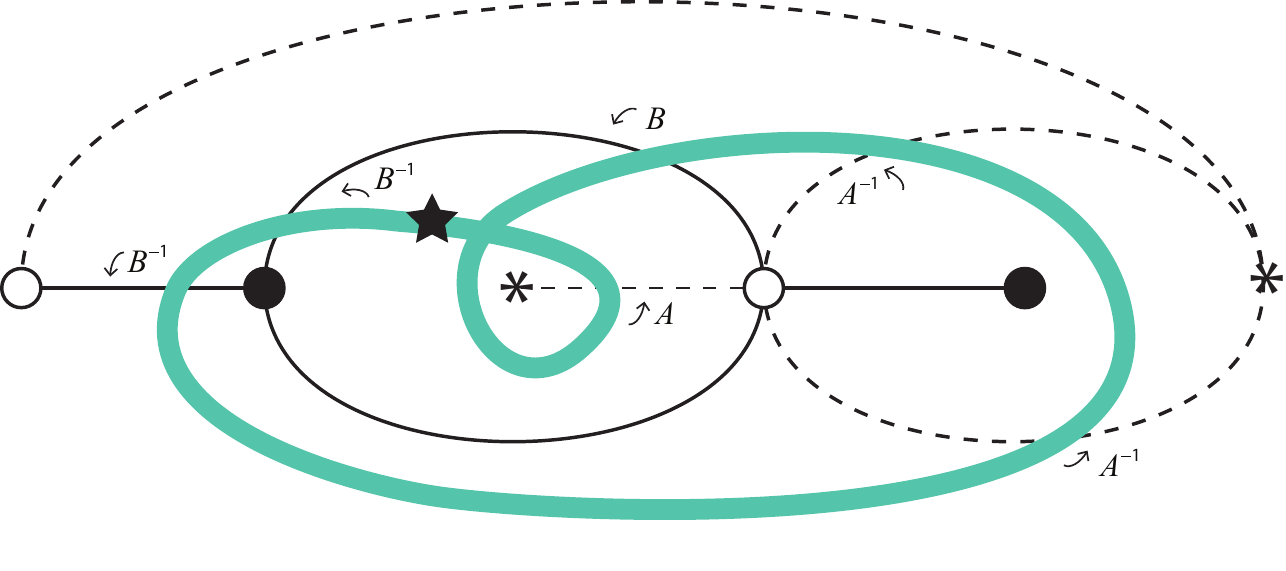}
\caption{The path $ABA^{-2}B^{-2}$ on the graph for $\Gamma_0(6) \cap \Gamma(2)$.}
\label{fig:gamma_06_tiled}
\end{figure}

If a graph of level $2n$ represents a noncongruence subgroup it suffices to find any element of $\Gamma(2n)$ which violates the criteria. As an example, consider the 3-star, which is also level 6, and thus if its associated $\Gamma$ is to be congruence it must contain $\Gamma(6)$. This graph does not violate the necessary conditions in Larcher's statements in Section \ref{larchers_results}, so we cannot immediately determine whether or not the subgroup is noncongruence. We see in Figure \ref{fig:3-star_noncongruence} the path corresponding to $ABA^{-2}B^{-2}$. Unlike the above example, this path does not form a loop. We conclude that the group $\Gamma = \left< BA^{-1}, \ A^{-1}B, \ B^{-1}A^{-1}B^2, \ B^{-3} \right>$ is a noncongruence subgroup of $\Gamma(2)$.
\begin{figure}[htb]
\centering
\includegraphics[width=.45 \textwidth]{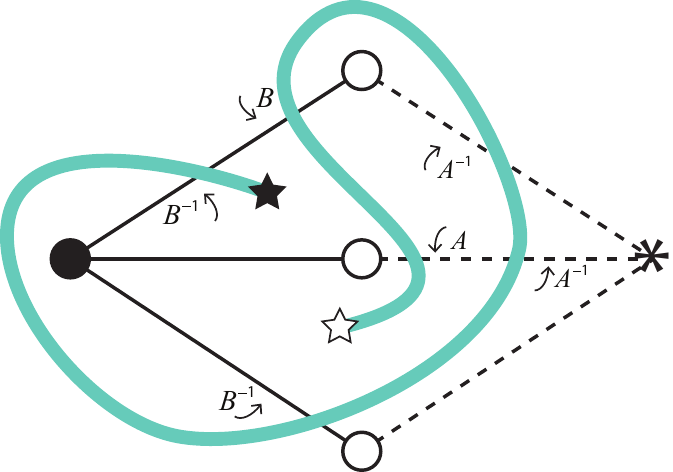}
\caption{The 3-star represents a noncongruence subgroup.}
\label{fig:3-star_noncongruence}
\end{figure}

For a final example, we use the graph on a surface of genus 2 pictured in Figure \ref{fig:level12example} to produce a noncongruence subgroup of Level 12. Its $\Gamma(2)$-tiling is shown in Figure \ref{fig:level12example_with_path}, along with a path for the element $BAB^3A^{-1}B^{-4} = {205 \, \, \, -24\choose 504 \, \, \, -59}$ in $\Gamma(12)$.
\begin{figure}[htb]
\centering
\includegraphics[width=.5 \textwidth]{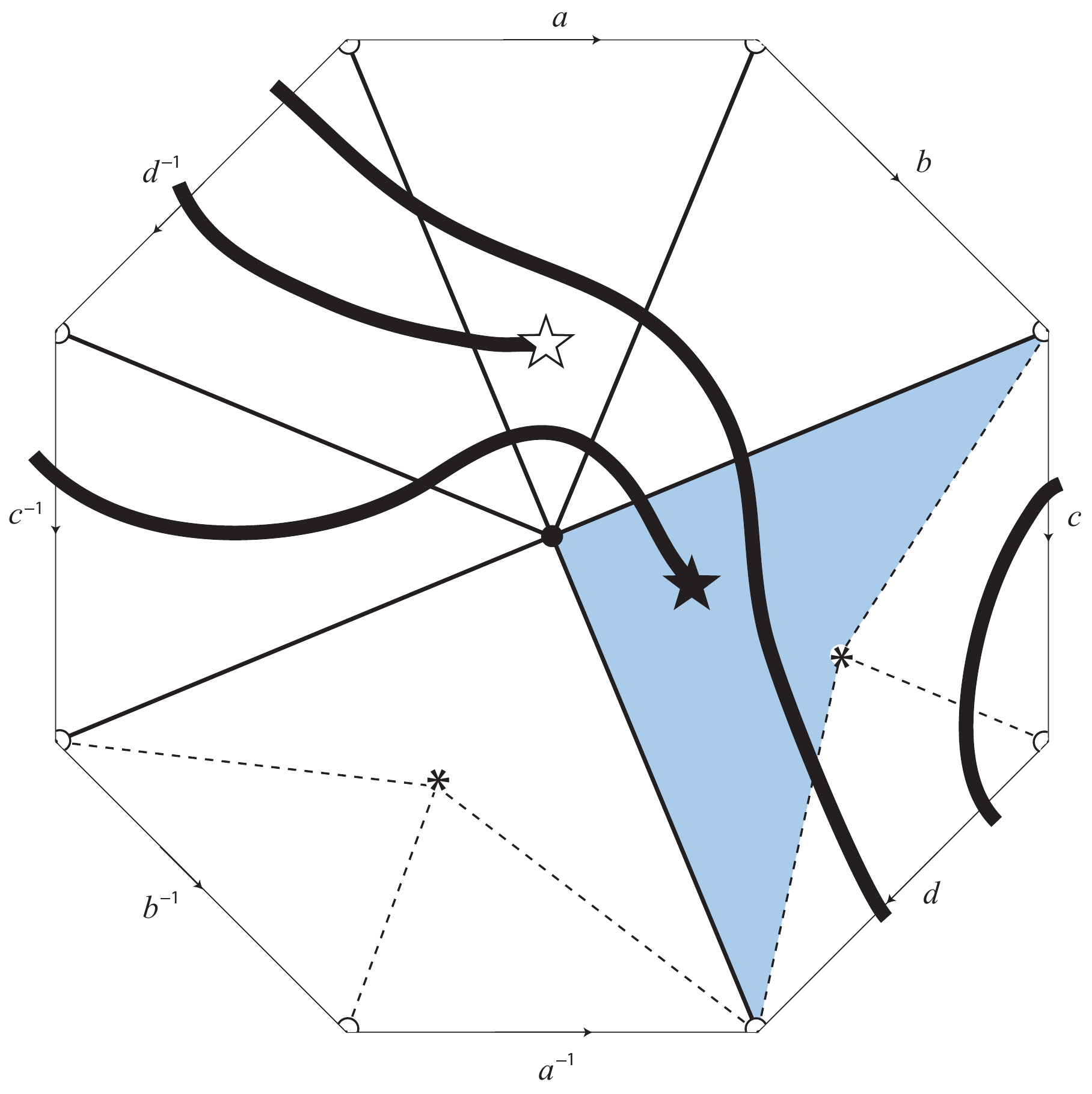}
\caption{The path $BAB^3A^{-1}B^{-4}$ on the graph in Figure \ref{fig:level12example}. }
\label{fig:level12example_with_path}
\end{figure}

Note that to prove a graph is congruence we need to work with a set of generators for $\Gamma(2n)$ given in terms of $A$ and $B$. While there are known ways to find generators for $\Gamma(n)$, none of the current methods we are aware of do so in terms of these generators for $\Gamma(2)$. In working toward this, we will see in a forthcoming paper an algorithm to produce the permutations for $\Gamma(2p)$ for $p$ prime. As we introduce more factors the generating sets become more complicated.


\end{document}